\documentclass[11pt]{amsart}

\usepackage{graphicx,geometry,stmaryrd,mathrsfs,mathtools,enumitem,bbm,hyperref,epstopdf,amsthm,booktabs,blkarray,amssymb}

\usepackage{framed}

\usepackage{mwe,tikz}
\usepackage[percent]{overpic}
\usepackage{relsize}
\usepackage[title]{appendix}

\newtheorem{thm}{Theorem}[section]

\newtheorem{lemma}[thm]{Lemma}

\theoremstyle{remark}

\theoremstyle{definition}
\newtheorem{definition}[thm]{Definition}

\newtheorem{assumption}[thm]{Assumption}

\numberwithin{equation}{section}

\newcommand{\dsp}{\displaystyle}

\makeatother
\newcommand\ackname{Acknowledgements}
\if@titlepage
\newenvironment{acknowledgements}{%
	\titlepage
	\null\vfil
	\@beginparpenalty\@lowpenalty
	\begin{center}%
		\bfseries \ackname
		\@endparpenalty\@M
\end{center}}%
{\par\vfil\null\endtitlepage}
\else

\fi
\makeatother


\begin{document}
\title[Instability of finite dipole]{On the existence and instability of solitary water waves with a finite dipole}

\author{Hung Le}
\address{Department of Mathematics, University of Missouri, Columbia, MO 65211}
\email{hdlgw3@mail.missouri.edu}

\date{\today}

\begin{abstract}	
This paper considers the existence and stability properties of two-dimensional solitary waves traversing an  infinitely deep body of water.  We assume that above the water is vacuum, and that the waves are acted upon by gravity with surface tension effects on the air--water interface.  In particular, we study the case where there is a finite dipole in the bulk of the fluid, that is, the vorticity is a sum of two weighted $\delta$-functions.  Using an implicit function theorem argument, we construct a family of solitary waves solutions for this system that is exhaustive in a neighborhood of $0$.  Our main result is that this family is conditionally orbitally unstable.  This is proved using a modification of the Grillakis--Shatah--Strauss method recently introduced by Varholm, Wahl\'en, and Walsh.  
\end{abstract}

\maketitle

\section{Introduction}

This paper is motivated by the following simple experiment.  Imagine that a surface water wave passes over a thin submerged body.  Boundary layer effects may then produce so-called \emph{shed vortices} --- highly localized vortical regions in the object's wake.  A natural idealization for this phenomenon is a \emph{finite dipole}, which is a weak solution of the Euler equations whose vorticity $\omega$ consists of a pair of Dirac $\delta$-measures (called point vortices) of nearly opposite strength that are separated by a fixed distance.

Dipoles are used commonly in fluid dynamical models; see further discussion in subsection \ref{subsection History of the problem}.  It is well-known that, if the problem is posed in the plane, then there are exact (stable) solutions for which the pair of vortices translate in parallel at a fixed velocity.  Here, we wish to study the far more complicated situation where the dipole lies inside a water wave.  We prove that there exists traveling wave solutions to this system.  However, our main result shows that they are conditionally orbitally unstable.  Physically, this indicates that a pair of counter-rotating shed vortices moving with a wave will not persist over long periods of time.  For instance, they may approach and then breach the surface.

\subsection{Main equations}

For each time $t \geq 0$, let $\Omega_t \subset \mathbb{R}^2$ be the fluid domain: 
\[
	\Omega_t := \left\{ (x_1,x_2) \in \mathbb{R}^2: x_2 < \eta(t,x_1) \right\},
\]
where the a priori unknown function $\eta=\eta(t,x_1)$ describes the free surface between air and water.  The water wave with a finite dipole problem is as follows.

Let $v=v(t,\cdot):\Omega_t \to \mathbb{R}^2$ be the fluid velocity.  The vorticity $\omega = \omega(t,\cdot) :\Omega_t \to \mathbb{R}$ is the (scalar) curl of $v$.  For a finite dipole, we must have 
\begin{equation} \label{eqn:vorticity w}
	\omega := \partial_{x_2} v_1 - \partial_{x_1}v_2 =  -\epsilon\gamma_1 \delta_{\bar x} + \epsilon\gamma_2 \delta_{\bar y},
\end{equation}
in the sense of distributions.  Here $\bar x = \bar x(t)$ and $\bar y = \bar y(t)$ are the vortex centers, and $\epsilon\gamma_1$ and $-\epsilon\gamma_2$ are the strengths, respectively.  We require that $v$ is a weak solution of the incompressible Euler equations away from the two point vortices:
\begin{align} \label{incompressible Euler eqn}
	\begin{dcases}
	\partial_t v + (v \cdot \nabla)v + \nabla p + ge_2 = 0 \quad \text{in } \Omega_t\backslash\{ \bar x,\bar y \}, \\
	\nabla \cdot v = 0 \quad \text{in } \Omega_t.
	\end{dcases} 
\end{align}
We require that there is finite excess kinetic energy, which corresponds to $v(t,\cdot) \in L^1_{\mathrm{loc}}(\Omega_t) \cap L^2(\Omega_t \backslash N_t)$ for any open set $N_t \subset \Omega_t \setminus
\{ \bar x, \bar y \}$.  

On the free surface $S_t:=\partial\Omega_t$, we have the kinematic and dynamic boundary condition:
\begin{align} \label{kinematic eqn}
	\partial_t \eta = -\eta' v_1 + v_2, \qquad p = b \kappa \quad \text{on } S_t,
\end{align}
where primes indicate derivatives with respect to $x_1$, and $\kappa = \kappa(t,x_1)$ is the mean curvature of the surface
\[
	\kappa(t,x_1) = -\frac{\eta''(t,x_1)}{\langle \eta'(t,x_1) \rangle^3}.
\]
Here we are using the Japanese bracket notation: $\langle \cdot \rangle := \left( 1 + (\cdot)^2 \right)^\frac{1}{2}$.  Moreover,  the constant $b > 0$ is the coefficient of surface tension.  

Finally, the motion of the vortices is governed by the Kirchhoff--Helmholtz model \cite{Helmholtz1858,Kirchhoff1850}:
\begin{align} \label{eqn vortices motion}
	\begin{dcases}
	\partial_t\bar{x} = \left.\left( v - \frac{\gamma_1}{2\pi} \epsilon \nabla^\perp \log |x-\bar{x}| \right)\right\vert_{\bar{x}}, \\
	\partial_t\bar{y} = \left.\left( v + \frac{\gamma_2}{2\pi} \epsilon \nabla^\perp \log |x-\bar{y}| \right)\right\vert_{\bar{y}},
	\end{dcases}
\end{align}
with $\nabla^\perp := (-\partial_{x_2},\partial_{x_1})$.  This system mandates that the point vortices are transported by the irrotational part of the fluid velocity field, and also attract each other due to the opposite vortex strengths.

\subsection{Statement of main results}

We are interested in both showing the existence of solitary waves solutions to \eqref{eqn:vorticity w}--\eqref{eqn vortices motion} and determining their stability.  As long as the two point vortices are separated from the surface, the fluid velocity $v$ can be decomposed as
\begin{align} \label{fluid velocity decompose 2 point vortices}
	v = \nabla\Phi + \epsilon \nabla\Theta
\end{align}
in a neighborhood of $S_t$, where $\Phi$ is a harmonic function and $\Theta$ represents the influence of the dipole.  Note that $\Theta$ can be written explicitly in terms of $\bar x$ and $\bar y$. To determine $v$, it is enough to know $\eta$ and the restriction of $\Phi$ to the surface $S_t$:
\begin{equation} \label{defn varphi}
	\varphi = \varphi(t,x_1) := \Phi\left( t, x_1, \eta(t,x_1) \right).
\end{equation}

For the steady problem, we look for solutions of the form
\[
	\eta = \eta^c(x_1-ct), \quad
	\varphi = \varphi^c(x_1-ct), \quad
	\bar{x} = cte_1 + (-a+\rho)e_2, \quad
	\bar{y} = cte_1 + (-a-\rho)e_2,
\]
where $(\eta^c, \varphi^c)$ are time-independent and spatially localized.  Specifically, we work in the space
\begin{equation} \label{space X for existence}
	(\eta, \varphi, a, \rho) \in X = X_1 \times X_2 \times X_3 \times X_4 := H^k_e(\mathbb{R}) \times \left( \dot{H}^{k-1}_o(\mathbb{R}) \cap \dot{H}^{1/2}_o(\mathbb{R}) \right) \times \mathbb{R} \times \mathbb{R}
\end{equation}
with
\[
	H^k_e(\mathbb{R}) := \{ f \in H^k(\mathbb{R}): f \text{ is even in } x_1 \},
	\quad
	H^k_o(\mathbb{R}) := \{ f \in H^k(\mathbb{R}): f \text{ is odd in } x_1 \},
\]
and let $\dot{H}^k_o(\mathbb{R})$ be the corresponding homogeneous space.   Then our first result is the existence of traveling capillary-gravity water waves with a finite dipole.  This theorem is an analogue of the work of Varholm on the water wave problem with point vortices in finite depth \cite{Varholm2016}.

\begin{thm}[Existence] 
\label{Theorem existence waves 2 point votices phantom diff vor str}
	
Let 
\[
	\bar x(t) = cte_1 + (-a_0+\rho_0)e_2, \qquad
	\bar y(t) = cte_1 + (-a_0-\rho_0)e_2.
\]
Then for every $k > \frac{3}{2}$, $a_0 \in (0,\infty)$, $\rho_0 \in (0,a_0)$, $\gamma_1^0 > 0$, and $\gamma_2^0 > 0$ subject to the compatibility condition
\begin{align} \label{eqn:compatibility condition}
	\gamma_2^0 = \frac{a_0^3+\rho_0^3}{a_0^3-\rho_0^3}\gamma_1^0,
\end{align}
there exists $\epsilon_1 > 0$, $c_1 > 0$, $\gamma_1^1>0$, $\gamma_2^1>0$, and $C^1$ family of traveling water wave with a finite dipole:
\begin{align*}
	\mathscr{C}_{\mathrm{loc}} &= \{ (\epsilon,c,\gamma_1,\gamma_2,\eta(\epsilon,c,\gamma_1,\gamma_2),\varphi(\epsilon,c,\gamma_1,\gamma_2),a(\epsilon,c,\gamma_1,\gamma_2),\rho(\epsilon,c,\gamma_1,\gamma_2)): \\
	&\hspace{2in} |\epsilon|<\epsilon_1, |c-c_0|<c_1, |\gamma_1-\gamma_1^0|<\gamma_1^1, |\gamma_2-\gamma_2^0|<\gamma_2^1 \} \\
	& \subset \mathbb{R} \times \mathbb{R} \times \mathbb{R} \times \mathbb{R} \times X.
\end{align*}
\end{thm}
Due to the variational structure of the problem, it is most natural to fix $\epsilon$, $\gamma_1$, $\gamma_2$, and consider the curve in $\mathscr{C}_{\mathrm{loc}}$ that results from varying the wave speed $c$.  This gives rise to a one-parameter family of solitary waves indexed by the wave speed:
\[
	U_c := (\eta(c),\varphi(c),\bar x(c),\bar y(c)).
\]

The compatibility condition \eqref{eqn:compatibility condition} implies that the lower vortex at $\bar{y}$ must have a greater strength than the upper vortex at $\bar{x}$, that is, $\gamma_2 > \gamma_1$ for $0<\rho_0<a_0$.  This is a consequence of the fact that $\bar{x}$ is closer to the free surface $S_t$ and is therefore influenced by it more strongly.  In contrast to finite dipoles in $\mathbb{R}^2$, which must have a total vorticity of $0$, the interaction with the wave in fact forces the point vortices to have slightly unmatched strengths.  It is worth mentioning that the compatibility condition \eqref{eqn:compatibility condition} is not artificial.  Indeed, as the family $\mathscr{C}_{\mathrm{loc}}$ is exhaustive in a neighborhood of $0$ in the space $X$, it must hold for any sufficiently small-amplitude, slow moving waves with even symmetry.  

Returning to the time-dependent problem, we introduce two important spaces.  Let
\begin{equation} \label{X space}
	\mathbb{X} := \mathbb{X}_1 \times \mathbb{X}_2 \times \mathbb{X}_3 \times \mathbb{X}_4 := H^1(\mathbb{R}) \times \dot{H}^{1/2}(\mathbb{R}) \times \mathbb{R}^2 \times \mathbb{R}^2,
\end{equation}
and set
\begin{equation} \label{W space}
	\mathbb{W} := \mathbb{W}_1 \times \mathbb{W}_2 \times \mathbb{W}_3 \times \mathbb{W}_4 := H^{3+}(\mathbb{R}) \times \left( \dot{H}^{5/2+}(\mathbb{R}) \cap \dot{H}^{1/2}(\mathbb{R}) \right) \times \mathbb{R}^2 \times \mathbb{R}^2,
\end{equation}
where $H^{k+}$ means $H^{k+s}$ for some fixed $0 < s \ll 1$.  We think of $\mathbb{W}$ as the \emph{well-posedness} space for \eqref{eqn:vorticity w}--\eqref{eqn vortices motion}.  A local well-posedness result for irrotational capillary-gravity water waves with this degree of regularity was proved by Alazard, Burq, and Zuily \cite{Alazard_Burq_Zuily2011}.  Very recently, Su \cite{Su2018} obtained local well-posedness for \eqref{eqn:vorticity w}--\eqref{eqn vortices motion} in a somewhat smoother space than $\mathbb{W}$; our results will also hold in that setting with only minor modifications.  On the other hand, $\mathbb{X}$ is the natural \emph{energy space}.  This is discussed in more detail in subsection \ref{subsection Hamiltonian formulation}.  Finally, for the problem to be well-defined, the finite dipole must be away from the free surface, so we take
\begin{align} \label{O space}
	\mathcal{O} := \left\{ u \in \mathbb{X}: \bar x_2 < \eta(\bar x_1) < -\bar x_2, \quad \bar y_2 < \eta(\bar y_1) < -\bar y_2, \quad \bar x \ne \bar y \right\}.
\end{align}

To state the main result, we introduce some terminology.  First, observe that the entire system is invariant under the one-parameter affine symmetry group $T(s):\mathbb{X} \to \mathbb{X}$ defined by
\begin{equation} \label{defn_symmetry_group_T}
	T(s)u := T(s)(\eta, \varphi, \bar{x}, \bar{y})^T = (\eta(\cdot - s), \varphi(\cdot - s), \bar{x} + se_1, \bar{y} + se_1)^T.
\end{equation}
This suggests that stability or instability should be understood modulo $T(s)$.   With that in mind, for each $\rho>0$, we define the tubular neighborhood
\[
	\mathcal{U}_\rho := \left\{ u \in \mathcal{O}: \inf_{s\in\mathbb{R}} \|T(s)U_c-u\|_{\mathbb{W}} < \rho \right\}.
\]
\begin{definition}
We say $U_c$ is \emph{orbitally unstable} provided that there is a $\nu_0 > 0$ such that for every $0 < \nu < \nu_0$ there exists initial data in $\mathcal{U}_{\nu}$ whose corresponding solution exits $\mathcal{U}_{\nu_0}$ in finite time.
\end{definition}
Our main theorem is as follows.

\begin{thm}[Instability] \label{Theorem Stability fixed dipole}
Assume that \eqref{eqn:vorticity w}--\eqref{eqn vortices motion} is locally well-posed in $\mathbb{W}$ in the sense of Assumption~\ref{Assumption local existence}.  For any $\epsilon \ne 0$ sufficiently small, the corresponding family of solitary capillary-gravity water waves with a finite dipole $U_c$ furnished by Theorem~\ref{Theorem existence waves 2 point votices phantom diff vor str} is conditionally orbitally unstable.  
\end{thm}

One physical interpretation for this is that, while we can construct steady configurations of counter rotating vortices moving in parallel through a water wave, these will not tend to persist over long periods of time.  Instead, we expect them to migrate to the surface of the water, fail to keep pace with the surface wave, or otherwise destabilize.  Moreover, this result covers all sufficiently small amplitude, wave speed, and vortex strength waves with even symmetry because $\mathscr{C}_{\mathrm{loc}}$ comprises also such waves near $0$ in $X$.

\subsection{History of the problem}
\label{subsection History of the problem}

The study of point vortices was initiated by Helmholtz \cite{Helmholtz1858} and Kirchhoff \cite{Kirchhoff1876}, who independently developed the model \eqref{eqn vortices motion}.  Since then, there has been extensive research on this subject.  The majority of this work concerns vortices in fixed fluid domains.  For instance, Love found a condition under which the motion of two pairs of vortices may be periodic \cite{Love1893-94} and investigated the stability of Kirchhoff's elliptic vortex \cite{Love1893-94Stability}.  Aref--Pomphrev \cite{Aref_Pomphrev1982} and Aref--Eckhardt \cite{Eckhardt_Aref1988} examined the chaotic behavior of the system of two pairs of vortices.  Wan \cite{Wan1988} proved the existence of steady concentrated vortex patches near the system of non-degenerate steady point vortices in the plane and on a bounded domain.
Marchioro and Pulvirenti \cite{Marchioro_Pulvirenti1993} later justified the connection between the incompressible Euler equation \eqref{incompressible Euler eqn} and the Kirchhoff--Helmholtz model \eqref{eqn vortices motion}.  Aref and Newton gave a thorough review of the results for $N$-vortex problem in the plane \cite{Aref2007,Newton2001} or on the surface of the sphere \cite{Newton2001}.  Recently, Smets--Van Schaftingen \cite{Smets_VanSchaftingen2010} and Cao--Liu--Wei \cite{Cao_Liu_Wei2014} studied the existence of solutions to the point vortex problem in a bounded domain using either a variational or Lyapunov-Schmidt reduction approach.  Kanso, Newton, and Tchieu also used the finite dipole as a model for fish schoolings \cite{Tchieu_Kanso_Newton2012}.  They examined the formation of multi-pole systems, discussed their stability, and compared the model against the real world scenario.  Point vortex models can also be used in studies of atmosphere and oceans \cite{Aref_Stremler2001}.

When a dipole is placed inside a water wave, which is the case in this paper, investigating existence and stability of solutions is much more involved mathematically as it requires developing an understanding of the interaction between the motion of the vortices and the free surface.  Nonetheless, there have been a sizable number of studies in this regime.  The first rigorous existence theory for steady solutions was given by Filippov \cite{Filippov1961} and Ter-Krikorov \cite{Ter-Krikorov1958}, who investigated the finite-depth regime neglecting surface tension.  Later, Shatah, Walsh, and Zeng constructed a family of traveling capillary gravity waves in infinite depth water with a single point vortex \cite{Shatah_Walsh_Zeng2013}.  Using a similar method, Varholm obtained analogues for capillary-gravity waves with one or more vortices in finite depth \cite{Varholm2016}.  Our existence theory follows in large part from the techniques in these two papers.

Our main source of inspiration is the recent paper by Varholm, Wahl\'en, and Walsh \cite{Varholm_Wahlen_Walsh2018} that proves the orbital stability of traveling capillary gravity waves with a single point vortex.  As we explain below, we will adopt a similar methodology.  However, the dipole turns out to be significantly more difficult to analyze at a technical level.  It is also of considerable importance to applications, as described above.  

It is well-known in the physics literature that the governing equations for water waves with submerged point vortices have a Hamiltonian structure.  Rouhi and Wright gave the formulation for the motion of vortices in the presence of a free surface in two and three dimensions \cite{Rouhi_Wright1993}.  A similar formulation was later given by Zakharov \cite{Zakharov1968}.  

There have also been a number of numerical results about vortex pairs in a fluid.  The closest to the current problem is the recent paper of Curtis, Carter, and Kalisch \cite{Curtis_Carter_Kalisch2018}, who studied how constant vorticity shear profile affects the motion of the particles both at and beneath waves in infinitely deep water.  Many authors have looked at the related scenario where a submerged dipole is sent moving towards the free surface rather than moving with the wave; see, for example, \cite{Telste1989}, \cite{Hirsa_Tryggvason_Willmart_Yu1989}, \cite{Tyvand1990}.  In all of these papers, the authors found cases where the vortices are able to breach the upper boundary.  The exact opposite scenario was considered by Su \cite{Su2018}, who proved that if a dipole initially moving away from the surface, the solution will persist over a long time scale.  This is in stark contrast to the present paper where we ask the dipole to move with the wave.  

\subsection{Plan of the article}

This paper contains two main sections.  In Section \ref{Section existence theory}, we show the existence of traveling capillary-gravity waves with a finite dipole.  This follows from an implicit function theorem argument in the spirit of Varholm \cite{Varholm2016} and Shatah, Walsh, and Zeng \cite{Shatah_Walsh_Zeng2013}.  Then, in Section \ref{Section stability theory}, we prove that these waves are orbitally unstable.  

We first establish that \eqref{eqn:vorticity w}--\eqref{eqn vortices motion} can be formulated as an infinite-dimensional Hamiltonian system of the form
\[
	\frac{\mathrm{d}u}{\mathrm{d}t} = J(u) DE(u)
\]
with $J$ being the Poisson map and $E$ the energy functional.  This is similar but distinct from the version due to Rouhi and Wright \cite{Rouhi_Wright1993}.  We offer a rigorous derivation in Theorem~\ref{Theorem Hamiltonian formulation}.  

In two seminal papers \cite{Grillakis_Shatah_Strauss1987,Grillakis_Shatah_Strauss1990},  Grillakis, Shatah, and Strauss provided a fairly simple method for determining the stability or instability of traveling wave solutions to systems of this form that are invariant under a continuous symmetry group.  Among the hypotheses of this theory are that the Poisson map $J$ is invertible, and that the initial value problem is globally well-posed in time.  Unfortunately, our $J$ is state-dependent and not surjective.  Moreover, we do not expect the problem to be well-posed in the energy space.  

In this paper, we will use a recent variant of the GSS method developed by  Varholm, Wahl\'en, and Walsh \cite{Varholm_Wahlen_Walsh2018}.  Among other improvements, this machinery permits $J$ to have merely a dense range, and also allow for a mismatch between the space where the problem is well-posed and the natural energy space.  In the present context, the latter point relates to the fact that $\mathbb{W} \nsubseteqq \mathbb{X}$.  As one of the hypotheses, we must compute the spectrum of the second variation of the augmented Hamiltonian defined in subsection~\ref{subsection spectrum}.  In particular, we show that it has a Morse index of $1$.  

For the convenience of the reader, several appendices have been included.  Appendix~\ref{Appendix Varholm--Wahlen--Walsh instability theory} contains a summary of the instability theory developed by Varholm--Wahl\'en--Walsh \cite{Varholm_Wahlen_Walsh2018}.  We also derive steady and unsteady equations for the velocity potential and stream functions in Appendix~\ref{Appendix steady and unsteady equations}.  Finally, in Appendix~\ref{Appendix variations of E and P}, we recorded the variations of the energy and momentum functional.

\section{Existence theory} \label{Section existence theory}

This section is devoted to proving the existence of traveling capillary-gravity water waves with a finite dipole.  We will adopt a methodology introduced by Varholm \cite{Varholm2016} and Shatah, Walsh, and Zeng \cite{Shatah_Walsh_Zeng2013}.  The first step is to reformulate \eqref{eqn:vorticity w}--\eqref{eqn vortices motion} in the spirit of Zakharov \cite{Zakharov1968}, and Craig and Sulem \cite{Craig_Sulem1993}.  This entails reducing the problem to a nonlocal system involving only surface variables.  

Recalling the splitting for $v$ in \eqref{fluid velocity decompose 2 point vortices}, we take $\Theta = \Theta_1 + \Theta_2 + \Theta_1^* + \Theta_2^*$ with
\begin{align} \label{Theta functions defn}
	\begin{split}
	\Theta_1(x) = -\frac{\gamma_1}{2\pi} \arctan \left( \frac{x_1-\bar{x}_1}{|x-\bar x| + x_2-\bar x_2} \right),
	\quad
	&\Theta_1^*(x) = \frac{\gamma_1}{2\pi} \arctan \left( \frac{x_1-\bar{x}_1^*}{|x-\bar x^*| + x_2-\bar{x}_2^*} \right),
	\\
	\Theta_2(x) = \frac{\gamma_2}{2\pi} \arctan \left( \frac{x_1-\bar{y}_1}{|x-\bar y| + x_2-\bar y_2} \right), 
	\quad
	&\Theta_2^*(x) = -\frac{\gamma_2}{2\pi} \arctan \left( \frac{x_1-\bar{y}_1^*}{|x-\bar y^*| + x_2-\bar{y}_2^*} \right),
	\end{split}
\end{align}
and $\bar{x}^* = (\bar{x}_1,-\bar{x}_2)$ and $\bar{y}^* = (\bar{y}_1,-\bar{y}_2)$ being the reflection of the two point vortices over the $x_1$-axis.  This corresponds to making a branch cut straight down from the vortex centers. It is easy to see that $v \in L^2 + L^2$ in the complement of any neighborhood of $\bar x$ and $\bar y$. 

It is often convenient to work with the harmonic conjugates of these functions.  In particular, let $\Gamma$ be the harmonic conjugate of $\Theta$, that is $\nabla\Theta=\nabla^{\perp}\Gamma$.  Then we have $\Gamma = \Gamma_1 + \Gamma_2 + \Gamma_1^* + \Gamma_2^*$, where
\begin{align*}
	\Gamma_1(x) &= \frac{\gamma_1}{2\pi} \log |x - \bar{x}|,
	\qquad
	\Gamma_2(x) = -\frac{\gamma_2}{2\pi} \log |x - \bar{y}|,
	\\
	\Gamma_1^*(x) &= -\frac{\gamma_1}{2\pi} \log |x - \bar{x}^*|,
	\qquad
	\Gamma_2^*(x) = \frac{\gamma_2}{2\pi} \log |x - \bar{y}^*|.
\end{align*}
We see that $-\Delta \Gamma_1 = -\gamma_1\delta_{\bar{x}}$, $-\Delta \Gamma_2 = \gamma_2\delta_{\bar{y}}$, and $-\Delta \Gamma = -\gamma_1\delta_{\bar{x}} + \gamma_2\delta_{\bar{y}}$, and hence
\[
	-\epsilon \Delta \Gamma = \omega.
\]
Define
\[
	\Xi_1 := \Theta_1 - \Theta_1^*, \quad
	\Xi_2 := \Theta_2 - \Theta_2^*, \quad
	\Upsilon_1 := \Theta_1 + \Theta_1^*, \quad
	\Upsilon_2 := \Theta_2 + \Theta_2^*,
\]
so that $\Theta = \Upsilon_1 + \Upsilon_2$.  This will be convenient for computing $\partial_{\bar x} \Theta$.  Also, let $\Psi$ be the harmonic conjugate of $\Phi$, so that
\[
	v = \nabla^{\perp}\Psi + \epsilon\nabla^\perp\Gamma,
\]
and denote the restriction of $\Psi$ to the surface $S_t$ by
\begin{equation} \label{defn psi}
	\psi = \psi(t,x_1) := \Psi(t,x_1,\eta(t,x_1)).
\end{equation}

We represent the normal derivatives of these functions on the free surface using the Dirichlet--Neumann operator $\mathcal{G}(\eta): \dot{H}^{1/2}(\mathbb{R}) \cap \dot{H}^k(\mathbb{R}) \to \dot{H}^{k-1}(\mathbb{R})$ defined by
\begin{align} \label{Dirichlet Neumann operator}
	\mathcal{G}(\eta) \phi := (-\eta' \partial_{x_1}\phi_\mathcal{H} + \partial_{x_2}\phi_\mathcal{H})|_{S_t},
\end{align}
where $\phi_\mathcal{H} \in \dot{H}^{k+1/2}(\Omega)$ is the harmonic extension of $\phi$ to $\Omega_t$ determined uniquely by
\[
	\Delta \phi_\mathcal{H} = 0 \text{ in } \Omega, \quad \phi_\mathcal{H} = \phi \text{ on } S_t,
\]
and $k \ge 0$.  It is well known that for any $\eta \in H^{k_0}(\mathbb{R})$, $k_0 > 3/2$, $\mathcal{G}(\eta)$ is a bounded, invertible, and self-adjoint operator between these spaces when $k \in [1-k_0, k_0]$.  Moreover, the mapping $\eta \mapsto \mathcal{G}(\eta)$ is $C^\infty$ and $\mathcal{G}(0) = |\partial_{x_1}|$ (see, for example, the book by Lannes \cite{Lannes2013}).

Next, we rewrite the water wave problem as the following system for the unknowns $(\eta,\varphi,\bar x,\bar y)$:
\begin{equation} \label{equation eta phi x_bar y_bar}
	\begin{dcases}
	\partial_t\eta &= \mathcal{G}(\eta)\varphi + \epsilon\nabla_\perp\Theta,
	\\
	\partial_t\varphi &= -\frac{1}{2\langle\eta'\rangle^2} \left( (\varphi')^2 - 2\eta'\varphi'\mathcal{G}(\eta)\varphi - (\mathcal{G}(\eta)\varphi)^2 \right) - \epsilon \partial_t \Theta - \epsilon \varphi' \partial_{x_1}\Theta - \frac{\epsilon^2}{2} |\nabla\Theta|^2 \\
	& \quad - \eta + b \frac{\eta''}{\langle \eta' \rangle^3},
	\\
	\partial_t\bar{x} &= \nabla\Phi(\bar{x}) + \epsilon\nabla\Theta_1^*(\bar{x}) + \epsilon\nabla\Theta_2(\bar{x}) + \epsilon\nabla\Theta_2^*(\bar{x}),
	\\
	\partial_t\bar{y} &= \nabla\Phi(\bar{y}) + \epsilon\nabla\Theta_1(\bar{y}) + \epsilon\nabla\Theta_1^*(\bar{y}) + \epsilon\nabla\Theta_2^*(\bar{y}).
	\end{dcases}
\end{equation}
Recall that $\varphi = \Phi|_{S_t}$ as in \eqref{defn varphi}, and  describes the irrotational part of the velocity field.  Here we have made use of the differential operators
\begin{align} \label{defn nabla_perp nabla_top}
	\nabla_\perp := \left.\left( -\eta'\partial_{x_1} + \partial_{x_2} \right)\right|_{S_t},
	\quad
	\nabla_\top := \left.\left( \partial_{x_1} + \eta'\partial_{x_2} \right)\right|_{S_t},
\end{align}
which come naturally as we take derivatives of functions restricted to the free surface.

Note that in \eqref{equation eta phi x_bar y_bar}, the equation for $\partial_t\eta$ can be derived from the kinematic boundary condition, but now $\Theta$ appears as a forcing term.  We can see that the evolution of $\varphi$ is determined by the unsteady Bernoulli equation \eqref{unsteady velocity potential equation}.  Finally, the equations for $\partial_t\bar{x}$ and $\partial_t\bar{y}$ come from the Kirchhoff--Helmholtz model \eqref{eqn vortices motion}.  

Now we are prepared to prove the existence theorem.  As this is done in the steady frame, we will simply write $S:=S_t$ and $\Omega:=\Omega_t$.

\begin{proof}[Proof of Theorem~\ref{Theorem existence waves 2 point votices phantom diff vor str}]
For convenience, we prove this result using $\psi$, which immediately gives the stated theorem in terms of $\varphi$.  For traveling waves solutions of \eqref{equation eta phi x_bar y_bar}, we have
\[
	\eta = \eta(x_1-ct), \quad
	\psi = \psi(x_1-ct), \quad
	\partial_t \bar{x} = ce_1, \quad
	\partial_t \bar{y} = ce_1.
\]
First we rescale: 
\[
	\eta =: \epsilon \tilde{\eta}, \quad \psi =: \epsilon \widetilde{\psi}, \quad \Psi =: \epsilon \widetilde{\Psi}, \quad c =: \epsilon \tilde{c},
\]
so that the steady point vortex motion equations \eqref{eqn vortices motion} now become
\[
	\begin{dcases}
	-\widetilde\Psi_{x_2}(\bar{x}) - \Gamma_{2_{x_2}}(\bar{x}) - \Gamma_{1_{x_2}}^*(\bar{x}) - \Gamma_{2_{x_2}}^*(\bar{x}) = \tilde{c}, \\
	-\widetilde\Psi_{x_2}(\bar{y}) - \Gamma_{1_{x_2}}(\bar{y}) - \Gamma_{1_{x_2}}^*(\bar{y}) - \Gamma_{2_{x_2}}^*(\bar{y}) = \tilde{c}.
	\end{dcases}
\]
Then the problem can be expressed as the abstract operator equation
\[
	\mathcal{F}(\epsilon,\tilde{c},\gamma_1,\gamma_2; \tilde{\eta},\widetilde{\psi},a, \rho) = 0
\]
with $\mathcal{F} = \left( \mathcal{F}_1,\mathcal{F}_2,\mathcal{F}_3,\mathcal{F}_4 \right): \mathbb{R} \times \mathbb{R} \times \mathbb{R} \times \mathbb{R} \times X \rightarrow Y$ given by
\begin{align} \label{defn F 2 point vortices}
	\mathcal{F}_1 &:= \frac{\epsilon\tilde{c}}{1+(\epsilon\tilde\eta')^2} \Big( \tilde\psi' + \epsilon\tilde\eta' \mathcal{G}(\epsilon\tilde\eta)\tilde\psi \Big) + \epsilon \tilde c \Gamma_{x_2}|_S + \frac{\epsilon}{2(1+(\epsilon\tilde\eta')^2)} \Big( (\tilde\psi')^2 + (\mathcal{G}(\epsilon\tilde\eta)\tilde\psi)^2 \Big) \nonumber\\
	&\hspace{.5in} + \frac{\epsilon}{1+(\epsilon\tilde\eta')^2} \Big( \mathcal{G}(\epsilon\tilde{\eta})\tilde\psi \nabla_\perp\Gamma + \tilde\psi'\nabla_\top\Gamma \Big) + \frac{\epsilon}{2} |(\nabla\Gamma)|_S|^2 + \tilde{\eta} + \frac{b}{\epsilon} \kappa(\epsilon \tilde{\eta}), \nonumber
	\\
	\mathcal{F}_2 &:= \epsilon \tilde{c} \tilde{\eta}' + \widetilde{\psi}' + (1,\epsilon \tilde{\eta}')^T \cdot \nabla \Gamma,
	\\
	\mathcal{F}_3 &:= \tilde{c} + \left(\partial_{x_2}\tilde{\psi}_\mathcal{H}\right)(\bar{x}) + \Gamma_{2_{x_2}}(\bar{x}) + \Gamma_{1_{x_2}}^*(\bar{x}) + \Gamma_{2_{x_2}}^*(\bar{x}), \nonumber
	\\
	\mathcal{F}_4 &:= \tilde{c} + \left(\partial_{x_2}\tilde{\psi}_\mathcal{H}\right)(\bar{y}) + \Gamma_{1_{x_2}}(\bar{y}) + \Gamma_{1_{x_2}}^*(\bar{y}) + \Gamma_{2_{x_2}}^*(\bar{y}), \nonumber
\end{align}
where $\nabla \Gamma$ is evaluated at $x_2 = \epsilon \tilde{\eta}(x_1)$ and $X$ is defined by \eqref{space X for existence}.  We take
\begin{align*}
	Y &= Y_1 \times Y_2 \times Y_3 \times Y_4 := H^{k-2}_e(\mathbb{R}) \times \left( \dot{H}^{k-2}_o(\mathbb{R}) \cap \dot{H}^{-1/2}_o(\mathbb{R}) \right) \times \mathbb{R} \times \mathbb{R}
\end{align*}
for $k > \frac{3}{2}$ fixed.

It is clear that $\mathcal{F}(\epsilon_0,\tilde{c}_0,\gamma_1^0,\gamma_2^0;\widetilde{\eta}_0,\widetilde{\psi}_0,a_0,\rho_0) = 0$ with 
\begin{subequations} \label{zero set of F}
\begin{align} 
\begin{split}
	\epsilon_0 &= 0,  \\
	\tilde{c}_0 &= -\Gamma_{2_{x_2}}(0,-a_0+\rho_0) - \Gamma_{1_{x_2}}^*(0,-a_0+\rho_0) - \Gamma_{2_{x_2}}^*(0,-a_0+\rho_0) \\
	&= -\frac{\gamma_1}{4\pi(a_0-\rho_0)} + \frac{\gamma_2}{4\pi} \left( \frac{1}{a_0} + \frac{1}{\rho_0} \right), 
\end{split}
\end{align}
and
\begin{align}
	\tilde\eta_0 = 0, \qquad
	\tilde\psi_0 = -\Gamma(x_1,0) = 0,
\end{align}
\end{subequations}
and $\gamma_1^0, \gamma_2^0, a_0, \rho_0 \in \mathbb{R}$ if and only if the compatibility condition \eqref{eqn:compatibility condition} holds.  A simple computation shows that
\begin{align*}
	\mathscr{L} &:= \left( D_{\tilde{\eta}},D_{\tilde{\psi}},\partial_a, \partial_\rho \right) \mathcal{F} \left( 0,\tilde{c}_0,\gamma_1^0,\gamma_2^0;0,\tilde{\psi}_0,a_0,\rho_0 \right) \\
	&=\begin{pmatrix}
	g-\alpha^2 \partial_{x_1}^2 & 0 & 0 & 0 \\
	0 & \partial_{x_1} & 0 & 0 \\
	0 & \left( \partial_{x_2} \langle \mathcal{H}(0),\cdot \rangle \right)|_{(0,-a_0+\rho_0)} & - \frac{\gamma_1^0}{4\pi(a_0-\rho_0)^2} + \frac{\gamma_2^0}{4\pi a_0^2} & \frac{\gamma_1^0}{4\pi(a_0-\rho_0)^2} + \frac{\gamma_2^0}{4\pi\rho_0^2} \\
	0 & \left( \partial_{x_2} \langle \mathcal{H}(0),\cdot \rangle \right)|_{(0,-a_0-\rho_0)} & - \frac{\gamma_1^0}{4\pi a_0^2} + \frac{\gamma_2^0}{4\pi(a_0+\rho_0)^2} & \frac{\gamma_2^0}{4\pi(a_0+\rho_0)^2} + \frac{\gamma_1^0}{4\pi\rho_0^2} \\
	\end{pmatrix} \\
	&\in \mathcal{L}(X,Y).
\end{align*}
The invertibility of $\mathscr{L}$ is equivalent to the invertibility of the $2\times 2$ real sub-matrix:
\begin{align} \label{matrix T}
	\mathscr{T} := \begin{pmatrix}
	\dsp -\frac{\gamma_1^0}{4\pi(a_0-\rho_0)^2} + \frac{\gamma_2^0}{4\pi a_0^2} & \dsp \frac{\gamma_1^0}{4\pi(a_0-\rho_0)^2} + \frac{\gamma_2^0}{4\pi\rho_0^2} \\
	\dsp -\frac{\gamma_1^0}{4\pi a_0^2} + \frac{\gamma_2^0}{4\pi(a_0+\rho_0)^2} & \dsp \frac{\gamma_2^0}{4\pi(a_0+\rho_0)^2} + \frac{\gamma_1^0}{4\pi\rho_0^2}
	\end{pmatrix}.
\end{align}
By the compatibility condition \eqref{eqn:compatibility condition}, we have
\[
	\det\mathscr{T} = -\frac{\gamma_1^2}{16\pi^2} \frac{6(a_0^4 - a_0^2 \rho_0^2 + \rho_0^4)}{(a_0 + \rho_0)(a_0 - \rho_0)^3(a_0^2 + a_0 \rho_0 + \rho_0^2)^2} < 0.
\]
Thus, $\mathscr{L}$ is an isomorphism.  The Implicit Function Theorem then tells us that there exists a family $\mathcal{C}_{\mathrm{loc}}$ of solutions of the form 
\[
	\mathcal{F}(\epsilon,\tilde{c},\gamma_1,\gamma_2;\tilde{\eta}(\epsilon,\tilde{c},\gamma_1,\gamma_2),\tilde{\psi}(\epsilon,\tilde{c},\gamma_1,\gamma_2),a(\epsilon,\tilde{c},\gamma_1,\gamma_2),\rho(\epsilon,\tilde{c},\gamma_1,\gamma_2)) = 0
\]
for all $|\epsilon| \ll 1$, $|\tilde{c}-\tilde{c}_0| \ll 1$, $|\gamma_1-\gamma_1^0| \ll 1$, and $|\gamma_2-\gamma_2^0| \ll 1$.  Theorem~\ref{Theorem existence waves 2 point votices phantom diff vor str} has, therefore, been proved.  Again, the Implicit Function Theorem allows us to infer that $\mathscr{C}_{\mathrm{loc}}$ comprises all traveling wave solutions in a neighborhood of $0$ in $X$.  
\end{proof}

For the stability analysis, we rely on asymptotic information about the traveling waves constructed above.  Using implicit differentiation, one can readily compute that  
\begin{subequations} \label{asymptotic forms traveling wave soln Uc}
\begin{align} \begin{split}
	\eta(\epsilon,\tilde{c},\gamma_1,\gamma_2) &= -\epsilon^2 (g - b \partial_{x_1}^2)^{-1} \left[ \tilde{c}_0\Gamma_{x_2}(x_1,0) \right] \\
	&\qquad + O ( |\epsilon|^3 + |\epsilon||c-c_0|^2 + |\epsilon||\gamma_1-\gamma_1^0|^2 + |\epsilon||\gamma_2-\gamma_2^0|^2 ) \quad \text{ in } C^1(U;X_1), \\
	\psi(\epsilon,\tilde{c},\gamma_1,\gamma_2) &= O ( |\epsilon|^3 + |\epsilon||c-c_0|^2 + |\epsilon||\gamma_1-\gamma_1^0|^2 + |\epsilon||\gamma_2-\gamma_2^0|^2 ) \quad \text{ in } C^1(U;X_2),
\end{split} \end{align}
and
\begin{align} \begin{split}
	a(\epsilon,\tilde{c},\gamma_1,\gamma_2) &= a_0 + \epsilon a_{0100}(c-c_0) + a_{0010}(\gamma_1-\gamma_1^0) + a_{0001}(\gamma_2-\gamma_2^0) \\
	&\qquad + O ( |\epsilon|^2 + |c-c_0|^2 + |\gamma_1-\gamma_1^0|^2 + |\gamma_2-\gamma_2^0|^2 ) \quad \text{ in } C^1(U;\mathbb{R}), \\
	\rho(\epsilon,\tilde{c},\gamma_1,\gamma_2) &= \rho_0 + \epsilon \rho_{0100}(c-c_0) + \rho_{0010}(\gamma_1-\gamma_1^0) + \rho_{0001}(\gamma_2-\gamma_2^0) \\
	&\qquad + O ( |\epsilon|^2 + |c-c_0|^2 + |\gamma_1-\gamma_1^0|^2 + |\gamma_2-\gamma_2^0|^2 ) \quad \text{in } C^1(U;\mathbb{R}),
\end{split} \end{align}
\end{subequations}
where $\widetilde{c}_0$, $\mathscr{T}$ are given by \eqref{zero set of F}--\eqref{matrix T}, and $U = B_{\epsilon_1}(0) \times B_{c_1}(c_0) \times B_{\gamma_1^1}(\gamma_1^0) \times B_{\gamma_2^1}(\gamma_2^0)$.  The indices $0100$, $0010$, and $0001$ are variations at the point $(0,\tilde{c}_0, \gamma_1^0, \gamma_2^0)$ with respect to $\tilde{c}$, $\gamma_1$, and $\gamma_2$, respectively.  In particular,
\begin{align} \label{a_tilde c rho_tilde c 2 point vortices}
\begin{split}
	a_{0100} &= \frac{1}{\det\mathscr{T}} \left( -\frac{\gamma_2^0}{4\pi(a_0+\rho_0)^2} - \frac{\gamma_1^0}{4\pi\rho_0^2} + \frac{\gamma_1^0}{4\pi(a_0-\rho_0)^2} + \frac{\gamma_2^0}{4\pi\rho_0^2} \right), \\
	\rho_{0100} &= \frac{1}{\det\mathscr{T}} \left( \frac{\gamma_2^0}{4\pi(a_0+\rho_0)^2} - \frac{\gamma_1^0}{4\pi a_0^2} + \frac{\gamma_1^0}{4\pi(a_0-\rho_0)^2} - \frac{\gamma_2^0}{4\pi a_0^2} \right).
\end{split}
\end{align}

\section{Instability theory} \label{Section stability theory}

In this section, we show that the traveling waves constructed in Theorem~\ref{Theorem existence waves 2 point votices phantom diff vor str} are orbitally unstable.  To do so, we follow the general strategy of Varholm--Wahl\'en--Walsh \cite{Varholm_Wahlen_Walsh2018} which is an adaptation of the classical Grillakis--Shatah--Strauss method \cite{Grillakis_Shatah_Strauss1987,Grillakis_Shatah_Strauss1990}.  In subsection \ref{subsection Hamiltonian formulation}, we rewrite the equations of capillary-gravity waves with a finite dipole \eqref{eqn:vorticity w}--\eqref{eqn vortices motion} as a Hamiltonian system and give an explicit form for its energy and momentum.  Next, in subsections~\ref{subsection spectrum} we prove that the spectrum of the second variation of the augmented Hamiltonian has the required configuration.  This is done in the spirit of Mielke \cite{Mielke2001}.  Finally, in subsection \ref{subsection proof of stability (main) theorem}, we complete the proof of our main result by computing the second derivative of the moment of instability for small waves in this family.  

\subsection{Hamiltonian formulation}
\label{subsection Hamiltonian formulation}

We first show that the system of equations \eqref{equation eta phi x_bar y_bar} has a Hamiltonian structure in terms of the state variable $u=(\eta,\varphi,\bar{x},\bar{y})^T$.  Define the energy $E=E(u)$ to be
\begin{equation} \label{energy functional}
	E(u) := K(u) + V(u),
\end{equation}
where $K$ is the (excess) kinetic energy and $V$ is the (excess) potential energy. The submerged dipole does not affect the latter, and so we may take 
\begin{align} \label{potential energy V}
	V(u) := \int_\mathbb{R} \left( \frac{1}{2} g\eta^2 + b( \langle\eta'\rangle - 1) \right) \,\mathrm{d}x_1.
\end{align}
However, some care is needed in the deriving the correct expression for $K$.  Formally, we take the classical kinetic energy $\frac{1}{2}\int_{\Omega} |v|^2 \,\mathrm{d}x$, split $v$ according to \eqref{fluid velocity decompose 2 point vortices}, and then integrate by parts.  The Newtonian potentials in $\Gamma$ will naturally lead to singular terms; these we discard.  What results is the following: 
\begin{align*} 
\begin{split}
	K(u) &:= K_0(u) + \epsilon K_1(u) + \epsilon^2 K_2(u) \\
	&= \frac{1}{2} \int_\mathbb{R} \varphi \mathcal{G}(\eta) \varphi \,\mathrm{d}x_1 + \epsilon \int_\mathbb{R} \varphi \nabla_\perp \Theta \,\mathrm{d}x_1 + \epsilon^2 \left( \frac{1}{2} \int_\mathbb{R} \Theta\vert_{S_t} \nabla_\perp \Theta \,\mathrm{d}x_1 + \Gamma^* \right),  \\
	\Gamma^* &:= \frac{\gamma_1}{2} \Big( \Gamma_1^*(\bar{x}) + \Gamma_2(\bar{x}) + \Gamma_2^*(\bar{x}) \Big) 
	- \frac{\gamma_2}{2} \Big( \Gamma_1(\bar{y}) + \Gamma_1^*(\bar{y}) + \Gamma_2^*(\bar{y}) \Big).
\end{split}
\end{align*}
Note that $K_0 = \frac{1}{2} \int_\Omega |\nabla\Phi|^2 \,\mathrm{d}x$, and hence represents the kinetic energy contributed by the purely irrotational part of the velocity.  $K_1$ is the interaction between the irrotational and rotational parts, and $K_2$ is the kinetic energy attributed to the rotational part.  

Recall the energy space $\mathbb{X}$ was defined by \eqref{X space} and the well-posedness space $\mathbb{W}$ was defined by \eqref{W space}.  As $\mathbb{X}$ is a Hilbert space, it is isomorphic to its continuous dual $\mathbb{X}^*$, and the isomorphism $I : \mathbb{X} \to \mathbb{X}^*$ takes the form
\[
	I = (1-\partial_{x_1}^2, |\partial_{x_1}|, \mathrm{Id}_{\mathbb{R}^2}, \mathrm{Id}_{\mathbb{R}^2}),
\]
where $\mathrm{Id}_{\mathbb{R}^2}$ is the $2 \times 2$ identity matrix.  For the Dirichlet--Neumann operator in $E$ to be well-defined, we want a smoother space than $\mathbb{X}$.  For that, we choose
\begin{equation} \label{V space}
	\mathbb{V} := \mathbb{V}_1 \times \mathbb{V}_2 \times \mathbb{V}_3 \times \mathbb{V}_4 := H^{3/2+}(\mathbb{R}) \times \left( \dot{H}^{1+}(\mathbb{R}) \cap \dot{H}^{1/2}(\mathbb{R}) \right) \times \mathbb{R}^2 \times \mathbb{R}^2.
\end{equation}
From the definition of the energy in \eqref{energy functional}, we see that $E \in C^{\infty}(\mathcal{O} \cap \mathbb{V}; \mathbb{R})$.  Using the Gagliardo--Nirenberg interpolation inequality, we can confirm that the spaces $\mathbb{X}$, $\mathbb{V}$, and $\mathbb{W}$ satisfy Assumption~\ref{Assumption space}.  In particular, there are constants $C > 0$ and $\theta \in (0,1/4)$ such that 
\[
	\|u\|_{\mathbb{V}}^3 \le C \|u\|_{\mathbb{X}}^{2+\theta} \|u\|_{\mathbb{W}}^{1-\theta}.
\]

Next, consider the closed operator $\widehat{J}:\mathcal{D}(\widehat{J}) \subset \mathbb{X}^* \to \mathbb{X}$ defined by
\begin{equation} \label{defn J hat}
\widehat{J} :=
\begin{pmatrix}
0  & 1 & 0                					& 0 \\
-1 & 0 & 0                					& 0 \\
0  & 0 & (\epsilon\gamma_1)^{-1}\mathcal{J} & 0 \\
0  & 0 & 0                					& -(\epsilon\gamma_2)^{-1}\mathcal{J}
\end{pmatrix},
\end{equation}
where $\mathcal{J}$ is a $2 \times 2$ real matrix
\[
	\mathcal{J} = \begin{pmatrix} 0 & -1 \\ 1 & 0 \end{pmatrix}
\]
and 
\[
	\mathcal{D}(\widehat{J}) = (H^{-1}(\mathbb{R}) \cap \dot{H}^{1/2}(\mathbb{R})) \times (H^1(\mathbb{R}) \cap \dot{H}^{-1/2}(\mathbb{R})) \times \mathbb{R}^2 \times \mathbb{R}^2.  
\]
Let $B \in C^1(\mathcal{O};\mathrm{Lin}(\mathbb{X})) \cap C^1(\mathcal{O} \cap \mathbb{W}; \mathrm{Lin}(\mathbb{W}))$ be defined by
\begin{equation} \label{defn B(u)}
	B(u) := \mathrm{Id}_{\mathbb{X}} + Z(u),
\end{equation}
where 
\[
	Z(u)\dot{w} :=
	\begin{pmatrix}
	0  &  0  &  0  &  0 \\
	-\epsilon(\gamma_1)^{-1} (\mathcal{J}\xi|_{S_t})^T  &  \epsilon(\gamma_2)^{-1} (\mathcal{J}\zeta|_{S_t})^T  &  \epsilon\xi^T|_{S_t}  &  \epsilon\zeta^T|_{S_t} \\
	\gamma_1^{-1} \mathcal{J}  &  0  &  0  &  0 \\
	0  & -(\gamma_2)^{-1} \mathcal{J} &  0  &  0
	\end{pmatrix}
	\begin{bmatrix}
		\langle \xi|_{S_t}, \dot{\eta} \rangle \\
		\langle \zeta|_{S_t}, \dot{\eta} \rangle \\
		\dot{\bar x} \\
		\dot{\bar y} 
	\end{bmatrix}
\]
for all $\dot w = (\dot\eta,\dot\varphi,\dot{\bar x},\dot{\bar y})^T \in \mathcal{O}$ with
\begin{align*}
	\xi := -\nabla_{\bar x}\Theta= (\Upsilon_{1_{x_1}},\Xi_{1_{x_2}})^T, \qquad
	\zeta := -\nabla_{\bar y}\Theta= (\Upsilon_{2_{x_1}},\Xi_{2_{x_2}})^T.
\end{align*}
The next lemma constructs the Poisson map $J$ in the Hamiltonian formulation and verifies that it satisfies Assumption~\ref{Assumption Poisson map}.

\begin{lemma}[Properties of $J$] \label{Lemma Properties of J}
For each $u \in \mathcal{O}$, let $J(u) : \mathcal{D}(\widehat{J}) \subset \mathbb{X}^* \to \mathbb{X}$ be defined by
\begin{align} \label{defn matrix J}
	J(u) := B(u) \widehat{J} = 
	\begin{pmatrix} 
		0  & 1      & 0      & 0      \\
		-1 & J_{22} & J_{23} 							 & J_{24} \\
		0  & J_{32} & (\epsilon\gamma_1)^{-1}\mathcal{J} & 0      \\
		0  & J_{42} & 0      							 & -(\epsilon\gamma_2)^{-1}\mathcal{J}
	\end{pmatrix},
\end{align}
where
\begin{multline*}
	J_{22} := \epsilon \Upsilon_{1_{x_1}}|_{S_t} \; \langle \cdot , -\gamma_1^{-1} \Xi_{1_{x_2}}|_{S_t} \rangle + \epsilon \Xi_{1_{x_2}}|_{S_t} \; \langle \cdot , \gamma_1^{-1} \Upsilon_{1_{x_1}}|_{S_t} \rangle \\
	+ \epsilon \Upsilon_{2_{x_1}}|_{S_t} \; \langle \cdot , \gamma_2^{-1} \Xi_{2_{x_2}}|_{S_t} \rangle 
	+ \epsilon \Xi_{2_{x_2}}|_{S_t} \; \langle \cdot , -\gamma_2^{-1} \Upsilon_{2_{x_1}}|_{S_t} \rangle,
\end{multline*}
\begin{flalign*}
	J_{23} := 
	\begin{pmatrix}
	\gamma_1^{-1}\Xi_{1_{x_2}}|_{S_t} , \quad -\gamma_1^{-1}\Upsilon_{1_{x_1}}|_{S_t}
	\end{pmatrix},
	\quad 
	J_{24} := 
	\begin{pmatrix}
	-\gamma_2^{-1}\Xi_{2_{x_2}}|_{S_t} , \quad \gamma_2^{-1}\Upsilon_{2_{x_1}}|_{S_t}
	\end{pmatrix}, &&
\end{flalign*}
\begin{flalign*}
	J_{32} := 
	\begin{pmatrix}
	\langle \cdot , -\gamma_1^{-1} \Xi_{1_{x_2}}|_{S_t} \rangle , \quad 
	\langle \cdot , \gamma_1^{-1} \Upsilon_{1_{x_1}}|_{S_t} \rangle
	\end{pmatrix}^T,
	\quad
	J_{42} :=
	\begin{pmatrix}
	\langle \cdot , \gamma_2^{-1} \Xi_{2_{x_2}}|_{S_t} \rangle , \quad
	\langle \cdot , -\gamma_2^{-1} \Upsilon_{2_{x_1}}|_{S_t} \rangle
	\end{pmatrix}^T. &&
\end{flalign*}
Then $J(u)$ satisfies Assumption~\ref{Assumption Poisson map}.
\end{lemma}

\begin{proof}
We see from its definition in \eqref{defn J hat} that $\widehat{J}$ is injective and its domain $\mathcal{D}(\widehat{J})$ is dense in $\mathbb{X}^* = I\mathbb{X}$, which proves that part (i) and (ii) of Assumption~\ref{Assumption Poisson map} are satisfied.  Moreover, since $B(u)$ defined by \eqref{defn B(u)} is both Fredholm index 0 and injective, it is an isomorphism on $\mathbb{X}$ and $\mathbb{W}$ giving part (iii).  Parts (iv) and (v) of Assumption~\ref{Assumption Poisson map} follow directly from the definitions \eqref{defn B(u)} and \eqref{defn matrix J}.
\end{proof}

In order to apply the instability theory in Appendix~\ref{Appendix Varholm--Wahlen--Walsh instability theory}, we must show that $DE$ can be realized as an element of $\mathbb{X}^*$.  With that in mind, we define the extension $\nabla E \in C^0(\mathcal{O} \cap \mathbb{V}; \mathbb{X}^*)$ by
\begin{multline} \label{defn nabla E}
	\langle \nabla E(u), v \rangle_{\mathbb{X}^* \times \mathbb{X}} 
	:= \langle E'_{\eta}(u), v_1 \rangle_{H^{-1} \times H^1} 
	+ \langle E'_{\varphi}(u), v_2 \rangle_{\dot{H}^{-1/2} \times \dot{H}^{1/2}} \\
	+ (E'_{\bar x}(u), v_3)_{\mathbb{R}^2}
	+ (E'_{\bar y}(u), v_4)_{\mathbb{R}^2},
\end{multline}
where
\[
	E'_{\varphi}(u) := \mathcal{G}(\eta)\varphi +\epsilon \nabla_\perp\Theta,
\]
\[
	 E'_{\eta}(u) := \frac{1}{2} \int_\mathbb{R} \varphi \langle D_\eta \mathcal{G}(\eta) \cdot, \varphi \rangle \,\mathrm{d}x_1 + \epsilon  \varphi'  \Theta_{x_1}|_{S_t}
	+ \frac{\epsilon^2}{2}  |(\nabla\Theta)|_{S_t}|^2  +  g\eta - b \left( \frac{\eta'}{\langle\eta'\rangle} \right)',
\]
\[
	E'_{\bar x}(u) := -\epsilon\int_{\mathbb{R}} \varphi\nabla_\perp\xi \,\mathrm{d}x_1 - \frac{\epsilon^2}{2} \int_{\mathbb{R}} \Big( \xi \nabla_\perp\Theta + \Theta|_{S_t} \nabla_\perp\xi \Big) \,\mathrm{d}x_1
	+ \nabla_{\bar x} \Gamma^*,
\]
\[
	E'_{\bar y}(u) := -\epsilon\int_{\mathbb{R}} \varphi\nabla_\perp\zeta \,\mathrm{d}x_1 - \frac{\epsilon^2}{2} \int_{\mathbb{R}} \Big( \zeta \nabla_\perp\Theta + \Theta|_{S_t} \nabla_\perp\zeta \Big) \,\mathrm{d}x_1 + \nabla_{\bar y} \Gamma^*.
\]
Here we use subscripts $x_1$ and $x_2$ to denote partial derivatives.  See Appendix~\ref{Appendix variations of E and P} for details.

\begin{thm}[Hamiltonian formulation] \label{Theorem Hamiltonian formulation}
The capillary-gravity water wave with a finite dipole problem \eqref{equation eta phi x_bar y_bar} has a solution $u := (\eta, \varphi, \bar{x}, \bar{y})^T \in C^1\left([0,t_0); \mathcal{O} \cap \mathbb{W}\right)$ if and only if it is a solution to the abstract Hamiltonian system
\begin{equation} \label{Hamiltonian abstract eqn}
	\frac{\mathrm{d}u}{\mathrm{d}t} = J(u) DE(u),
\end{equation}
where $E$ is the energy functional defined in \eqref{energy functional} and $J$ is the skew-symmetric operator defined by \eqref{defn matrix J}.
\end{thm}

\begin{proof}
Throughout the proof, we make repeated use of the identities
\begin{align} \label{identities grad perp and grad T}
	\nabla_\perp\psi_{x_1} = \nabla_\top\psi_{x_2} = \Big( \psi_{x_2}|_{S_t} \Big)',
	\quad
	\nabla_\perp\psi_{x_2} = -\nabla_\top\psi_{x_1} = -\Big( \psi_{x_1}|_{S_t} \Big)',
\end{align}
where $\psi$ is any function harmonic in a neighborhood of $S_t$, and recall $\nabla_\perp$ and $\nabla_\top$ are defined in \eqref{defn nabla_perp nabla_top}.

Suppose we have a solution $u$ of \eqref{Hamiltonian abstract eqn}.  From the expressions for $J$ in \eqref{defn matrix J} and the differential equation \eqref{Hamiltonian abstract eqn}, we see that
\[
	\partial_t \eta = D_\varphi E(u) = \mathcal{G}(\eta) \varphi + \epsilon \nabla_\perp \Theta,
\]
which is the kinematic condition \eqref{kinematic eqn}.  

Next, we verify that 
\[
	\partial_t \bar{x} = J_{32}\varphi + (\epsilon \gamma_1)^{-1} \mathcal{J} \nabla_{\bar x}E(u)
\]
is equivalent to the ODE for $\bar x$ in \eqref{eqn vortices motion}.  Explicitly, the first component of the equation is
\begin{equation} \label{Claim partial_t xbar1}
	\partial_t \bar{x}_1 = -(\epsilon \gamma_1)^{-1} \partial_{\bar{x}_2} E(u) + \left\langle D_\varphi E(u), -\gamma_1^{-1} \Xi_{1_{x_2}}|_{S_t} \right\rangle.
\end{equation}
Using the fact that
\[
	\nabla_\top f = \partial_{x_1} \Big( f|_S \Big), \quad \nabla_\perp \Phi = \nabla_\top \Psi, \quad \nabla_\top \Phi = -\nabla_\perp \Psi, \quad \nabla_\perp \Theta = \nabla_\top \Gamma, \quad \nabla_\top \Theta = -\nabla_\perp \Gamma,
\]
and the identities \eqref{identities grad perp and grad T},  \eqref{Claim partial_t xbar1} becomes
\begin{align*}
	\partial_t\bar x &= \frac{1}{\gamma_1} \int_{\mathbb{R}} \left( -\Xi_{1_{x_1}}|_{S_t} \nabla_\perp\Psi + \Psi|_{S_t} \nabla_\perp\Xi_{1_{x_1}} \right) \,\mathrm{d}x_1 \\ 
	& \qquad  + \frac{\epsilon}{2\gamma_1} \int_{\mathbb{R}} \left( -\Xi_{1_{x_1}}|_{S_t} \nabla_\perp\Gamma - \Xi_{1_{x_2}}|_{S_t} \nabla_\perp\Theta \right) \,\mathrm{d}x_1  - \epsilon \Gamma_{1_{x_2}}^*(\bar{x}) - \epsilon \Gamma_{2_{x_2}}(\bar{x}) - \epsilon \Gamma_{2_{x_2}}^*(\bar{x}) \\
	& =: \frac{1}{\gamma_1} \mathscr{A} + \frac{\epsilon}{2\gamma_1} \mathscr{B} - \epsilon \Gamma_{1_{x_2}}^*(\bar{x}) - \epsilon \Gamma_{2_{x_2}}(\bar{x}) - \epsilon \Gamma_{2_{x_2}}^*(\bar{x}).
\end{align*}
Since $\Psi$ and $\Theta_1^*$ are harmonic in $\Omega_t$, for any $0 < r \ll 1$ we have
\begin{align*}
	\mathscr{A} &= -\int_{\partial B_r(\bar{x})} \Big( -\Theta_{1_{x_1}} N\cdot\nabla\Psi + \Psi N\cdot\nabla\Theta_{1_{x_1}} \Big) \,\mathrm{d}S_t \\
	&= -\int_{\partial B_r(\bar{x})} \Big( \frac{\gamma_1}{2\pi} \frac{x_2-\bar{x}_2}{|x-\bar{x}|^2} N\cdot\nabla\Psi + \Psi \frac{\gamma_1}{2\pi} \frac{x_2-\bar{x}_2}{|x-\bar{x}|^3} \Big) \,\mathrm{d}S_t \\
	&= -\frac{\gamma_1}{2\pi} \int_0^{2\pi} \Big( \frac{r\sin\theta}{r^2} \partial_r (\Psi) + \Psi \frac{r\sin\theta}{r^3} \Big) r \,\mathrm{d}\theta.
\end{align*}
Expanding $\Psi$ around $r=0$ gives
\begin{align*}
	\mathscr{A} &= -\frac{\gamma_1}{2\pi} \int_0^{2\pi} \Big[ \sin\theta(\Psi_{x_1} \cos\theta + \Psi_{x_2} \sin\theta) + \frac{\sin\theta}{r}(\Psi(\bar{x}) + \Psi_{x_1}(\bar{x}) r\cos\theta \\
	& \qquad + \Psi_{x_2}(\bar{x}) r\sin\theta) \Big] \,\mathrm{d}\theta + o(r)
	= -\gamma_1 \Psi_{x_2}(\bar{x}) + o(r)
\end{align*}
as $r \to 0$.  A direct computation along the same lines shows $\mathscr{B} = 0$.  Thus, \eqref{Claim partial_t xbar1} is equivalent to 
\[
	\partial_t \bar x_1 = -\Psi_{x_2}(\bar x) - \epsilon \Gamma_{1_{x_2}}^*(\bar{x}) - \epsilon \Gamma_{2_{x_2}}(\bar{x}) - \epsilon \Gamma_{2_{x_2}}^*(\bar{x}),
\]
which agrees the Kirchhoff--Helmholtz model \eqref{eqn vortices motion}.  By nearly identical arguments, we likewise confirm that the same holds for $\partial_t \bar x_2$ and then $\partial_t \bar y$.

Finally, we claim that
\begin{align} \label{partial varphi Hamitonian explicit form}
\begin{split}
	\partial_t \varphi = -D_\eta E(u) 
	&+ \xi|_{S_t} \cdot (\gamma_1^{-1}\mathcal{J}) \nabla_{\bar x} E(u)
	+ \zeta|_{S_t} \cdot (\gamma_2^{-1}\mathcal{J}) \nabla_{\bar y} E(u)
	\\
	&+ \epsilon \xi|_{S_t} \left\langle D_{\varphi}E(u), (-\gamma_1^{-1}\mathcal{J}) \xi \right\rangle
	+ \epsilon \zeta|_{S_t} \left\langle D_{\varphi}E(u), (-\gamma_2^{-1}\mathcal{J}) \zeta \right\rangle
\end{split}
\end{align}
is equivalent to the unsteady Bernoulli condition in \eqref{equation eta phi x_bar y_bar}.  By a well-known formula for the derivative $\mathcal{G}(\eta)$ (see, for example, \cite[Proposition 2.1]{Mielke2001}), we know that
\[
	\int_\mathbb{R} \varphi \langle D_\eta \mathcal{G}(\eta) \dot{\eta}, \varphi \rangle \,\mathrm{d}x_1 = \int_\mathbb{R} \frac{1}{\langle\eta'\rangle^2} \Big( (\varphi')^2 - (\mathcal{G}(\eta)\varphi)^2 - 2 \eta' \varphi' \mathcal{G}(\eta)\varphi \Big) \dot{\eta} \,\mathrm{d}x_1.
\]
Then
\begin{align*}
	&\partial_t \varphi = -\frac{1}{\langle\eta'\rangle^2} \Big( (\varphi')^2 - (\mathcal{G}(\eta)\varphi)^2 - 2 \eta' \varphi' \mathcal{G}(\eta)\varphi \Big) + \epsilon \varphi' \Gamma_{x_2}|_{S_t} - \frac{\epsilon^2}{2} |(\nabla \Theta)|_{S_t}|^2 - V'_\eta(u) \\
	&\qquad + \epsilon \Theta_{1_{x_1}}|_{S_t} \; \partial_t \bar{x}_1 + \epsilon \Theta_{1_{x_2}}|_{S_t} \; \partial_t \bar{x}_2 + \epsilon \Theta_{2_{x_1}}|_{S_t} \; \partial_t \bar{y}_1 + \epsilon \Theta_{2_{x_2}}|_{S_t} \; \partial_t \bar{y}_2.  
\end{align*}
Here we have used the fact that for $\Theta = (\Theta_1 + \Theta_1^* + \Theta_2 + \Theta_2^*)(x_1, x_2, \bar{x}, \bar{y})$,
\begin{align*}
	( \partial_t \Theta )|_S &= (-\Theta_{1_{x_1}}-\Theta_{1_{x_1}}^*)|_{S_t} \, \partial_t \bar{x}_1 + (-\Theta_{1_{x_2}}+\Theta_{1_{x_2}}^*)|_{S_t} \, \partial_t \bar{x}_2 + (-\Theta_{2_{x_1}}-\Theta_{2_{x_1}}^*)|_{S_t} \, \partial_t \bar{y}_1 \\
	&\qquad + (-\Theta_{2_{x_2}}+\Theta_{2_{x_2}}^*)|_{S_t} \, \partial_t \bar{y}_2 \\
	&= -\Upsilon_{1_{x_1}}|_{S_t} \; \partial_t \bar{x}_1 - \Xi_{1_{x_2}}|_{S_t} \; \partial_t \bar{x}_2 - \Upsilon_{2_{x_1}}|_{S_t} \; \partial_t \bar{y}_1 - \Xi_{2_{x_2}}|_{S_t} \; \partial_t \bar{y}_2.
\end{align*}
Thus, comparing this to the equations for $\varphi$ in \eqref{equation eta phi x_bar y_bar}, the claim has been proved.  
\end{proof}

The momentum associated to a solution of the system \eqref{Hamiltonian abstract eqn} is given by
\begin{align} \label{momentum P}
	P = P(u) = -\epsilon \gamma_1 \bar{x}_2 + \epsilon \gamma_2 \bar{y}_2 - \int_\mathbb{R} \eta' (\varphi + \epsilon \Theta|_{S_t}) \,\mathrm{d}x_1.
\end{align}
It is clear that $P \in C^\infty(\mathcal{O} \cap \mathbb{V}; \mathbb{R})$.  Similarly to the Fr\'echet derivatives of the energy, $DP$ can be extended to $\nabla P \in C^0(\mathcal{O} \cap \mathbb{V}; \mathbb{X}^*)$:
\begin{multline} \label{defn nabla P}
	\langle \nabla P(u),v \rangle_{\mathbb{X}^* \times \mathbb{X}} 
	:= \langle P'_{\eta}(u),v_1 \rangle_{H^{-1} \times H^1}
	+ \langle P'_{\varphi}(u),v_2 \rangle_{\dot{H}^{-1/2} \times \dot{H}^{1/2}} \\
	+ (P'_{\bar x}(u),v_3)_{\mathbb{R}^2}
	+ (P'_{\bar y}(u),v_4)_{\mathbb{R}^2}
\end{multline}
with
\begin{align*}
	P'_{\eta}(u) &= \varphi' + \epsilon\Theta_{x_1}|_{S_t},
	& P'_{\varphi}(u) &= -\eta',
	\\
	P'_{\bar x}(u) &:= -\epsilon\gamma_1 e_2 + \epsilon\int_{\mathbb{R}} \eta' \xi|_{S_t} \,\mathrm{d}x_1,
	& P'_{\bar y}(u) &:= \epsilon\gamma_2 e_2 + \epsilon\int_{\mathbb{R}} \eta' \zeta|_{S_t} \,\mathrm{d}x_1.
\end{align*}
It is immediate from their definitions in \eqref{defn nabla E} and \eqref{defn nabla P} that $\nabla E$ and $\nabla P$ satisfy Assumption \ref{Assumption derivative extension}.  Observe also that $\nabla P$ is in $\mathcal{D}(\widehat{J})$ and 
\begin{align} \label{expression JDP}
	J(u) \nabla P(u) = \left( -\eta', -\varphi', 1, 0, 1, 0 \right)^T.
\end{align}
The next lemma records the fact that the momentum and the energy are conserved.  

\begin{lemma}[Conservation]
Suppose that $u \in C^0 \left( [0,t_0); \mathcal{O} \cap \mathbb{W} \right)$ is a distributional solution to the Cauchy problem \eqref{Hamiltonian abstract eqn} with initial data $u_0 \in \mathcal{O} \cap \mathbb{W}$.  Then
\[
	E(u(t)) = E(u_0)
	\quad \text{and} \quad
	P(u(t)) = P(u_0)
	\quad \text{for all } t \in [0,t_0).
\]
\end{lemma}
The proof follows directly from computation and the regularity of the well-posedness space $\mathbb{W}$; see, \cite[Theorem 5.3]{Varholm_Wahlen_Walsh2018}.  

Next, we verify that the symmetry group $T$ defined by \eqref{defn_symmetry_group_T} indeed satisfies Assumption~\ref{Assumption symmetry group}.  The linear part of $T$ is 
\begin{equation} \label{defn:dT(s)}
	dT(s)u = (\eta(\cdot - s), \varphi(\cdot - s), \bar x, \bar y)^T
	\quad \textrm{for all } s \in \mathbb{R}, u \in \mathbb{X},
\end{equation}
and the infinitesimal generator of $T$ is the unbounded affine operator
\begin{equation} \label{defn_infinitesimal generator T}
	T'(0)u := (-\eta',-\varphi',e_1,e_1)^T
	\quad \textrm{for all } u \in \mathcal{D}(T'(0))
\end{equation}
with $\mathcal{D}(T'(0)) = H^2(\mathbb{R}) \times \left( \dot{H}^{3/2}(\mathbb{R}) \cap \dot{H}^{1/2}(\mathbb{R}) \right) \times \mathbb{R}^2 \times \mathbb{R}^2$.

\begin{lemma} \label{Lemma properties of T}
The group $T(s)$ satisfies Assumption~\ref{Assumption symmetry group}.
\end{lemma}
The proof of this lemma is done by nearly identical arguments to that of  \cite[Lemma 5.4]{Varholm_Wahlen_Walsh2018}, as the symmetry groups are essentially the same.  We therefore omit the details.

Finally, recall that 
\begin{equation} \label{Uc:with_I}
	\left\{ U_c = (\eta(c),\varphi(c),\bar{x}(c),\bar{y}(c)) : c \in \mathcal{I}\right \}
\end{equation}
is a one-parameter family of solitary capillary-gravity water waves with a finite dipole constructed in Theorem~\ref{Theorem existence waves 2 point votices phantom diff vor str}, where we fix $\gamma_1$, $\gamma_2$, $\epsilon$, and vary the wave speed $c$.  Here $\mathcal{I}$ is an open interval containing 
\[
	c_0 = -\frac{\epsilon \gamma_1}{4\pi (a_0-\rho_0)} + \frac{\epsilon \gamma_2}{4\pi} \left( \frac{1}{a_0}+\frac{1}{\rho_0} \right).
\]

The next lemma confirms that this family satisfies Assumption~\ref{Assumption bound states} of the general stability theory.  

\begin{lemma}
Fix any choice of $0 < \rho_0 < a_0$, and consider the corresponding surface of solutions $\mathscr{C}_{\mathrm{loc}}$ furnished by Theorem~\ref{Theorem existence waves 2 point votices phantom diff vor str}.  Then the corresponding one-parameter family of bound states $\{ U_c \}_{c \in \mathcal{I}}$  satisfies Assumption~\ref{Assumption bound states}.
\end{lemma}

\begin{proof}
From the proof of Theorem~\ref{Theorem existence waves 2 point votices phantom diff vor str} in Section~\ref{Section existence theory}, it is clear that $\mathscr{C}_{\mathrm{loc}}$ is in fact $C^\infty$, thus $c \mapsto U_c$ is likewise smooth, and part (i) of the assumption holds.  The asymptotic form of the solutions given in \eqref{asymptotic forms traveling wave soln Uc} immediately shows that part (ii) holds.  Part (iii) follows from the fact that the existence theory can be carried out in $H^k$ spaces for any $k > \frac{3}{2}$.  Finally, as these are solitary waves, it is obvious that the second alternative of part (iv) is satisfied.    
\end{proof}

\subsection{Spectrum of the augmented potential} 
\label{subsection spectrum}

We define the augmented Hamiltonian to be
\[
	E_c(u) := E(u) - cP(u).
\]
The \emph{moment of instability} is the scalar-valued function that results from evaluating $E_c$ along the family $\{U_c\}$:
\begin{equation} \label{defn moment of instability}
	d(c) := E_c(U_c) = E(U_c) - cP(U_c).
\end{equation}
By \eqref{Hamiltonian abstract eqn} and \eqref{expression JDP}, we have $JDE(U_c) - cJDP(U_c) = 0$, and hence
\begin{align} \label{DE_c}
	DE_c(U_c) = DE(U_c) - cDP(U_c) = 0.
\end{align}
Thus, each traveling wave $U_c$ is a critical point of the augmented Hamiltonian.  Then differentiating $d$ gives the identity
\[
	d'(c) = \left\langle DE(U_c) - cDP(U_c), \frac{\mathrm{d}U_c}{\mathrm{d}c} \right\rangle - P(U_c) = -P(U_c).
\]
If we differentiate \eqref{DE_c} with respect to $c$, we also obtain
\[
	\left\langle D^2 E_c(U_c) \frac{\mathrm{d}U_c}{\mathrm{d}c}, \; \cdot \; \right\rangle = \left\langle DP(U_c), \; \cdot \; \right\rangle.
\]
As in the work of Grillakis, Shatah, and Strauss \cite{Grillakis_Shatah_Strauss1987,Grillakis_Shatah_Strauss1990}, we must show that the linearized Hamiltonian has Morse index $1$.  That is, $D^2 E_c(U_c)$ can be associated to a bounded self-adjoint operator on $\mathbb{X}$ whose spectrum takes the form $\{-\mu_c^2\} \cup \{0\} \cup \Sigma_c$, where $\Sigma_c \subset \mathbb{R}_+$ is uniformly bounded away from 0 and $-\mu_c^2 < 0$.  This corresponds to Assumption~\ref{Assumption spectrum}.

We first note that $0$ is in the spectrum.  Indeed, for all $s \in \mathbb{R}$, $T(s)U_c$ is also a traveling wave solution.  Therefore,
\[
	DE_c(T(s)U_c) = 0
\]
for all $s$.  Differentiating with respect to $s$ gives
\[
	\left\langle D^2 E_c(T(s)U_c), T'(0)U_c \right\rangle = 0,
\]
and hence, $T'(0)U_c$ is an eigenfunction for eigenvalue $0$.

Following Mielke's approach \cite{Mielke2001}, we will determine the remaining spectrum by first considering the \emph{augmented potential}
\begin{equation} \label{V_c defn}
	V_c = V_c(\eta,\bar{x},\bar{y}) := \min_{\varphi \in \mathbb{V}_2} E_c(\eta,\varphi,\bar{x},\bar{y}) =: E_c(\eta,\varphi_m,\bar{x},\bar{y})
\end{equation}
for $(\eta,\bar x,\bar y) \in \mathbb{V}_1 \times \mathbb{V}_3 \times \mathbb{V}_4$, which corresponds to fixing $(\eta,\bar{x},\bar{y})$ and minimizing $E_c$ over $\varphi$.  Thus,
\begin{equation} \label{eqn:D_phi E_c = 0}
	D_\varphi E_c(\eta,\varphi_m,\bar{x},\bar{y}) = 0.
\end{equation}
Because $\varphi$ occurs quadratically in the energy, it is easy to see that this minimum is attained exactly when
\begin{align} \label{varphi_m}
	\varphi_m(\eta,\bar{x},\bar{y}) = \mathcal{G}(\eta)^{-1}[-c\eta'-\epsilon\nabla_\perp\Theta].
\end{align}

Since we will be doing many calculations where $\varphi$ is fixed, we adopt the notational convention that for $u = (u_1,u_2,u_3,u_4)$,
\[
	v := (u_1,u_3,u_4)
\]
and write a variation in the direction $v$ as $\dot v$.  We also use the short hand $\mathbb{V}_{1,3,4} := \mathbb{V}_1 \times \mathbb{V}_3 \times \mathbb{V}_4$, and for convenience, define 
\[
	u_m(v):=(\eta,\varphi_m,\bar x,\bar y) \in \mathbb{V}.
\]
Also, when we evaluate derivatives of the Dirichlet--Neumann operator, we will encounter the quantities 
\[
	\mathfrak{a} := (\nabla(\mathcal{H}\varphi_m))|_S, \quad
	\mathfrak{b} := \mathfrak{a} + \epsilon (\nabla\Theta)|_S - ce_1.
\]
See Appendix~\ref{Appendix steady and unsteady equations} for an explicit formula giving $\mathfrak{a}$ in terms of $\varphi$ and $\eta$.  Physically, $\mathfrak{b}$ is the restriction of the full relative velocity to the interface.  Therefore, by the kinematic condition, $\mathfrak{b}_2 = \eta' \mathfrak{b}_1$; this also follows directly from \eqref{varphi_m}.

While it is not completely obvious, we will see that the spectral properties of $D^2 E_c(U_c)$ can be inferred from those of $D^2 V_c(v)$.  With that in mind, the first step is to derive a formula for the second variation of the augmented potential.  

\begin{lemma} \label{Lemma:D^2V_c:defn_mathcal_L}
For all $v \in \mathbb{V}_{1,3,4} \cap \mathcal{O}_{1,3,4}$ and all variations $\dot v \in \mathbb{V}_{1,3,4}$, we have
\begin{multline} \label{D^2 mathcal V aug c simple}
	\left\langle D^2 V_c(v)\dot v, \dot v \right\rangle_{\mathbb{V}^*_{1,3,4} \times \mathbb{V}_{1,3,4}} = - \left\langle \mathcal{L}(v) \dot v, \mathcal{G}(\eta)^{-1} \mathcal{L}(v) \dot v \right\rangle_{\mathbb{X}_2^* \times \mathbb{X}_2} \\
	+ \left\langle D^2_v E_c(u_m(v)) \dot v, \dot v \right\rangle_{\mathbb{V}^*_{1,3,4} \times \mathbb{V}_{1,3,4}},
\end{multline}
where $\mathcal{L}(v) \in \mathrm{Lin}(\mathbb{X}_{1,3,4};\mathbb{X}^*_2)$ defined by
\begin{align} \label{defn mathcal L}
	\mathcal{L}(v) \dot v := \mathcal{G}(\eta)(\mathfrak{a}_2\dot\eta) + (\mathfrak{b}_1\dot\eta)' + \epsilon\nabla_\perp\xi \cdot \dot{\bar x} + \epsilon\nabla_\perp\zeta \cdot \dot{\bar y}.
\end{align}
\end{lemma}
The proof follows by a straightforward adaptation of \cite[Lemma 6.2]{Varholm_Wahlen_Walsh2018}, and we therefore omit it.  In the next lemma, we refine expression \eqref{D^2 mathcal V aug c simple} to derive a quadratic form representation of $D^2 V_c$.

\begin{lemma}[Quadratic form] \label{Lemma D^2 V_c^aug quadratic form}
For all $v \in \mathbb{V}_{1,3,4} \cap \mathcal{O}_{1,3,4}$, there is a self-adjoint linear operator $A(v) \in \mathrm{Lin}(\mathbb{X}_{1,3,4};\mathbb{X}_{1,3,4}^*)$ such that
\[
	\left\langle D^2 V_c(v) \dot v, \dot w \right\rangle_{\mathbb{V}_{1,3,4}^* \times \mathbb{V}_{1,3,4}} = \left\langle A \dot v, \dot w \right\rangle_{\mathbb{X}_{1,3,4}^* \times \mathbb{X}_{1,3,4}}
\]
for all $\dot v, \dot w \in \mathbb{V}_{1,3,4}$.  The form of $A$ is given in \eqref{A expression}.  
\end{lemma}

\begin{proof}
From \cite[Proposition 2.1]{Mielke2001}, we have
\[
	\int_{\mathbb{R}} \hat{\varphi} \langle D_\eta \mathcal{G}(\eta)\dot{\eta},\varphi \rangle \,\mathrm{d}x_1 = \int_{\mathbb{R}} \dot{\eta} (\mathfrak{a}_1 \hat{\varphi}' - \mathfrak{a}_2 \mathcal{G}(\eta)\hat{\varphi}) \,\mathrm{d}x_1
\] 
and
\[
	\int_{\mathbb{R}} \varphi \langle \langle D^2_\eta \mathcal{G}(\eta)\dot{\eta}, \dot{\eta} \rangle, \varphi \rangle \,\mathrm{d}x_1 = 2 \int_{\mathbb{R}} \Big( \dot{\eta}^2 \mathfrak{a}_1' \mathfrak{a}_2 + \mathfrak{a}_2 \dot{\eta} \mathcal{G}(\eta)(\mathfrak{a}_2 \dot{\eta}) \Big) \,\mathrm{d}x_1.
\]
Letting the self-adjoint operator $\mathcal{M}$ defined by
\[
	\mathcal{M}\dot\eta := -\mathfrak{b}_1(\mathcal{G}(\eta)^{-1} (\mathfrak{b}_1\dot{\eta})')',
\]
and using the fact that $\mathcal{G}(\eta)^{-1}$ is a self-adjoint operator, we can compute
\begin{multline*}
	\int_{\mathbb{R}} \mathcal{L}(v)\dot v \mathcal{G}(\eta)^{-1} \mathcal{L}(v)\dot v \,\mathrm{d}x_1 = \int_{\mathbb{R}} \mathfrak{a}_2 \dot\eta \mathcal{G}(\eta)(\mathfrak{a}_2 \eta) \,\mathrm{d}x_1 + \int_{\mathbb{R}} \dot\eta\mathcal{M}\dot\eta \,\mathrm{d}x_1 + \int_{\mathbb{R}} (\mathfrak{a}_2\mathfrak{b}_1' - \mathfrak{a}_2'\mathfrak{b}_1) \dot\eta^2 \,\mathrm{d}x_1 \hfill \\
	+ 2\epsilon\dot{\bar x} \cdot \int_{\mathbb{R}} \left( \mathfrak{a}_2\nabla_\perp\xi - \mathfrak{b}_1 \left( \mathcal{G}(\eta)^{-1}\nabla_\perp\xi \right)' \right) \dot\eta \,\mathrm{d}x_1 
	+ \epsilon^2 \dot{\bar x}^T \left( \int_{\mathbb{R}} \nabla_\perp\xi \odot \mathcal{G}(\eta)^{-1}\nabla_\perp\xi \,\mathrm{d}x_1 \right) \dot{\bar x} \\
	+ 2\epsilon\dot{\bar y} \cdot \int_{\mathbb{R}} \left( \mathfrak{a}_2\nabla_\perp\zeta - \mathfrak{b}_1 \left( \mathcal{G}(\eta)^{-1}\nabla_\perp\zeta \right)' \right) \dot\eta \,\mathrm{d}x_1 
	+ \epsilon^2 \dot{\bar y}^T \left( \int_{\mathbb{R}} \nabla_\perp\zeta \odot \mathcal{G}(\eta)^{-1}\nabla_\perp\zeta \,\mathrm{d}x_1 \right) \dot{\bar y} \\
	+ 2\epsilon^2 \dot{\bar x}^T \left( \int_{\mathbb{R}} \nabla_\perp\xi \odot \mathcal{G}(\eta)^{-1}\nabla_\perp\zeta \,\mathrm{d}x_1 \right) \dot{\bar y},
\end{multline*}
where $x \odot y = (x \otimes y + y \otimes x)/2$ is the symmetric outer product. Next, we have
\[
	\left\langle D^2_\eta E_c(u_m)\dot\eta, \dot\eta \right\rangle = \int_{\mathbb{R}} \mathfrak{a}_2 \dot\eta \mathcal{G}(\eta)(\mathfrak{a}_2 \eta) \,\mathrm{d}x_1 + \int_{\mathbb{R}} \left( g + \epsilon\mathfrak{b}_1\nabla_\top\Theta_{x_2} + \mathfrak{a}_2\mathfrak{b}_1' \right) \dot\eta^2 \,\mathrm{d}x_1 + \int_{\mathbb{R}} \frac{b}{\langle\eta'\rangle^3} (\dot\eta')^2 \,\mathrm{d}x_1,
\]
\begin{flalign*}
	\nabla_{\bar x} \langle D_\eta E_c(u_m), \dot\eta \rangle = \epsilon \int_{\mathbb{R}} (\mathfrak{a}_2 \nabla_\perp\xi - \mathfrak{b}_1 \nabla_\top\xi) \dot\eta \,\mathrm{d}x_1, &&
\end{flalign*}
\begin{flalign*}
	\nabla_{\bar y} \langle D_\eta E_c(u_m), \dot\eta \rangle = \epsilon \int_{\mathbb{R}} (\mathfrak{a}_2 \nabla_\perp\zeta - \mathfrak{b}_1 \nabla_\top\zeta) \dot\eta \,\mathrm{d}x_1, &&
\end{flalign*}
\begin{multline*}
	D^2_{\bar x} E_c(u_m) = 2\epsilon^2 D^2_{\bar x} \Gamma^* - \epsilon\int_{\mathbb{R}} (\mathcal{G}(\eta)\varphi_m D^2_{\bar x}\Theta + \varphi_m' D^2_{\bar x}\Gamma)|_S \,\mathrm{d}x_1 + \epsilon^2\int_{\mathbb{R}} \nabla_\perp\xi \odot \xi \,\mathrm{d}x_1 \\
	- \frac{\epsilon^2}{2} \int_{\mathbb{R}} (\nabla_\perp\Theta D^2_{\bar x}\Theta + \nabla_\top\Theta D^2_{\bar x}\Gamma)|_S \,\mathrm{d}x_1,
\end{multline*}
\begin{multline*}
	D^2_{\bar y} E_c(u_m) = 2\epsilon^2 D^2_{\bar y} \Gamma^* - \epsilon\int_{\mathbb{R}} (\mathcal{G}(\eta)\varphi_m D^2_{\bar y}\Theta + \varphi_m' D^2_{\bar y}\Gamma)|_S \,\mathrm{d}x_1 + \epsilon^2\int_{\mathbb{R}} \nabla_\perp\zeta \odot \zeta \,\mathrm{d}x_1 \\
	- \frac{\epsilon^2}{2} \int_{\mathbb{R}} (\nabla_\perp\Theta D^2_{\bar y}\Theta + \nabla_\top\Theta D^2_{\bar y}\Gamma)|_S \,\mathrm{d}x_1,
\end{multline*}
\begin{flalign*}
	\nabla_{\bar x}\nabla_{\bar y} E_c(u_m) = \epsilon^2 \nabla_{\bar x}\nabla_{\bar y} \Gamma^* + \frac{\epsilon^2}{2} \int_{\mathbb{R}} \nabla_\perp (\xi \odot \zeta) \,\mathrm{d}x_1. &&
\end{flalign*}
Substituting the above results into the expression \eqref{D^2 mathcal V aug c simple}, we arrive at
\begin{multline*}
	\langle D^2 V_c(v) \dot v,\dot v \rangle = \int_{\mathbb{R}} (g + \mathfrak{b}_2' \mathfrak{b}_1) \dot\eta^2 \,\mathrm{d}x_1 - \int_{\mathbb{R}} \left( \frac{b}{\langle\eta'\rangle^3} \dot \eta' \right)' \dot\eta \,\mathrm{d}x_1 - \int_{\mathbb{R}} \dot\eta\mathcal{M}\dot\eta \,\mathrm{d}x_1 \\
	+ 2\epsilon\dot{\bar x} \cdot \int_{\mathbb{R}} \dot\eta \mathfrak{b}_1 \nabla_\top(\mathcal{G}(\eta)^{-1}\nabla_\perp\xi - \xi) \,\mathrm{d}x_1 
	+ 2\epsilon\dot{\bar y} \cdot \int_{\mathbb{R}} \dot\eta \mathfrak{b}_1 \nabla_\top(\mathcal{G}(\eta)^{-1}\nabla_\perp\zeta - \zeta) \,\mathrm{d}x_1 \\
	+ \dot{\bar x}^T \left( D_{\bar x}^2 E_c(u_m) - \epsilon^2 \int_{\mathbb{R}} \nabla_\perp\xi \odot \mathcal{G}(\eta)^{-1}\nabla_\perp\xi \,\mathrm{d}x_1 \right) \dot{\bar x} \\
	+ \dot{\bar y}^T \left( D_{\bar y}^2 E_c(u_m) - \epsilon^2 \int_{\mathbb{R}} \nabla_\perp\zeta \odot \mathcal{G}(\eta)^{-1}\nabla_\perp\zeta \,\mathrm{d}x_1 \right) \dot{\bar y} \\
	+ \dot{\bar x}^T \left( \nabla_{\bar x}\nabla_{\bar y} E_c(u_m) - 2\epsilon^2 \int_{\mathbb{R}} \nabla_\perp\xi \odot \mathcal{G}(\eta)^{-1}\nabla_\perp\zeta \,\mathrm{d}x_1 \right) \dot{\bar y}.
\end{multline*}

Thus, inspecting the above formula, we see that the claimed quadratic form representation holds with the operator $A$ defined as follows:

\begin{subequations} \label{A expression} 
\begin{align} \label{A11:expression}
	A_{11}\dot\eta := (g + \mathfrak{b}_2' \mathfrak{b}_1)\dot\eta - \left( \frac{b}{\langle\eta'\rangle^3}\dot\eta' \right)' - \mathcal{M}\dot\eta,	
\end{align}
\begin{align} \label{A13:expression}
	A_{13}\dot{\bar x} := \epsilon\mathfrak{b}_1 \nabla_\top(\mathcal{G}(\eta)^{-1}\nabla_\perp\xi-\xi) \cdot \dot{\bar x},
\end{align}
\begin{align} \label{A31:expression}
	A_{13}^*\dot\eta := \epsilon \int_{\mathbb{R}} \dot\eta \mathfrak{b}_1 \nabla_\top(\mathcal{G}(\eta)^{-1}\nabla_\perp\xi-\xi) \,\mathrm{d}x_1,
\end{align}
\begin{align} \label{A14:expression}
	A_{14}\dot{\bar y} := \epsilon\mathfrak{b}_1 \nabla_\top(\mathcal{G}(\eta)^{-1}\nabla_\perp\zeta-\zeta) \cdot \dot{\bar y},
\end{align}
\begin{align} \label{A41:expression}
	A_{14}^*\dot\eta := \epsilon \int_{\mathbb{R}} \dot\eta \mathfrak{b}_1 \nabla_\top(\mathcal{G}(\eta)^{-1}\nabla_\perp\zeta-\zeta) \,\mathrm{d}x_1,
\end{align}
\begin{align} \label{A33:expression}
	A_{33} := D_{\bar x}^2 E_c(u_m) - \epsilon^2 \int_{\mathbb{R}} \nabla_\perp\xi \odot \mathcal{G}(\eta)^{-1}\nabla_\perp\xi \,\mathrm{d}x_1,
\end{align}
\begin{align} \label{A44:expression}
	A_{44} := D_{\bar y}^2 E_c(u_m) - \epsilon^2 \int_{\mathbb{R}} \nabla_\perp\zeta \odot \mathcal{G}(\eta)^{-1}\nabla_\perp\zeta \,\mathrm{d}x_1,
\end{align}
\begin{align} \label{A34,A43:expression}
	A_{34} = A_{43} := \nabla_{\bar x}\nabla_{\bar y} E_c(u_m) - \epsilon^2 \int_{\mathbb{R}} \nabla_\perp\xi \odot \mathcal{G}(\eta)^{-1}\nabla_\perp\zeta \,\mathrm{d}x_1.
\end{align}
\end{subequations}
This finishes the proof of Lemma~\ref{Lemma D^2 V_c^aug quadratic form}.
\end{proof}

The next lemma verifies that the second variation of the augmented Hamiltonian $E_c$ has an extension to the energy space $\mathbb{X}$.

\begin{lemma}[Extension of $D^2 E_c$] \label{Lemma extension D^2 Ec}
For all $v \in \mathbb{V}_{1,3,4} \cap \mathcal{O}_{1,3,4}$, there exists a self-adjoint operator $H_c(v) \in \mathrm{Lin}(\mathbb{X},\mathbb{X}^*)$ such that
\begin{align} \label{extension D^2 Ec identity}
	\langle D^2 E_c(u_m(v)) \dot u, \dot w \rangle_{\mathbb{V}^* \times \mathbb{V}} = \langle H_c(v) \dot u, \dot w \rangle_{\mathbb{X}^* \times \mathbb{X}}
\end{align}
for all $\dot u, \dot w \in \mathbb{V}$ with
\[
	H_c(v) \dot u =
	\begin{pmatrix}
	\mathrm{Id}_{\mathbb{X}_1^*}  &  0  &  0  &  0  \\
	0  &  0  &  0  &  \mathrm{Id}_{\mathbb{X}_2^*}  \\
	0  &  \mathrm{Id}_{\mathbb{R}^2}  &  0  &  0  \\
	0  &  0  &  \mathrm{Id}_{\mathbb{R}^2}  &  0 
	\end{pmatrix}
	\begin{pmatrix}
	A(v) + \mathcal{L}(v)^* \mathcal{G}(\eta)^{-1}\mathcal{L}(v)  &  -\mathcal{L}(v)^* \\
	-\mathcal{L}(v)  &  \mathcal{G}(\eta)
	\end{pmatrix}
	\begin{bmatrix}	\dot v \\ \dot \varphi \end{bmatrix},
\]
where $\mathcal{L}(v)$ and $A(v)$ are defined in Lemmas~\ref{Lemma:D^2V_c:defn_mathcal_L} and~\ref{Lemma D^2 V_c^aug quadratic form}, respectively.  The adjoint $\mathcal{L}(v)^* \in \mathrm{Lin}(\mathbb{X}_2; \mathbb{X}_{1,3,4}^*)$ is given by
\[
	\mathcal{L}(v)^* \dot\varphi = (\mathfrak{a}_2\mathcal{G}(\eta)\dot\varphi - \mathfrak{b}_1 \dot\varphi', \epsilon \langle \nabla_\perp(\xi+\zeta),\dot\varphi \rangle),
\]
and we have
\begin{equation} \label{eqn:H_c}
	\langle H_c(v) \dot u,\dot u \rangle_{\mathbb{X}^* \times \mathbb{X}} = \langle A(v) \dot v,\dot v \rangle_{\mathbb{X}_{1,3,4}^* \times \mathbb{X}_{1,3,4}} + \left\langle \mathcal{G}(\eta)(\dot\varphi-\mathcal{G}(\eta)^{-1}\mathcal{L}\dot v), \dot\varphi-\mathcal{G}(\eta)^{-1} \mathcal{L} \dot v \right\rangle_{\mathbb{X}_2^* \times \mathbb{X}_2}
\end{equation}
for all $\dot u \in \mathbb{X}$.
\end{lemma}

\begin{proof}
It is straightforward to see that
\[
	\langle D_\varphi D_v E_c(u_m(v)) \dot v,\dot\varphi \rangle_{\mathbb{V}_2^* \times \mathbb{V}_2} = - \int_{\mathbb{R}} \dot v \mathcal{L}(v) \dot v \,\mathrm{d}x_1
\]
holds for all $\dot v \in \mathbb{V}_{1,3,4}$ and $\dot\varphi \in \mathbb{V}_2$.  Because of symmetry, it suffices to consider only the diagonal entries.  For all $\dot u \in \mathbb{V}$, Lemmas~\ref{Lemma:D^2V_c:defn_mathcal_L} and \ref{Lemma D^2 V_c^aug quadratic form} give
\begin{align} \label{D2 E_c Lemma Extension}
\begin{split}
	\langle D^2 E_c(u_m(v)) \dot u , \dot u \rangle_{\mathbb{V}^* \times \mathbb{V}} & =  \langle D_v^2 E_c(u_m(v)) \dot v, \dot v \rangle + 2 \langle D_\varphi D_v E_c(u_m(v)) \dot v, \dot\varphi \rangle \\
	& \qquad + \langle D_\varphi^2 E_c(u_m(v)) \dot\varphi, \dot\varphi \rangle \\
	&= \langle A(v) \dot v, \dot v \rangle_{\mathbb{X}_{1,3,4}^* \times \mathbb{X}_{1,3,4}} \\ 
	& \qquad + \int_{\mathbb{R}} \Big[ (\mathcal{L}(v) \dot v) \mathcal{G}(\eta)^{-1} \mathcal{L}(v) \dot v - 2 \dot\varphi \mathcal{L}(v) \dot v + \dot\varphi \mathcal{G}(\eta) \dot\varphi \Big] \,\mathrm{d}x_1 \\
	&= \langle A(v) \dot v, \dot v \rangle_{\mathbb{X}_{1,3,4}^* \times \mathbb{X}_{1,3,4}} - \int_{\mathbb{R}} \mathcal{L}(v) \dot v  (\dot\varphi - \mathcal{G}(\eta)^{-1} \mathcal{L}(v) \dot v) \,\mathrm{d}x_1 \\
	& \qquad + \int_{\mathbb{R}}  \dot\varphi \mathcal{G}(\eta) (\dot\varphi - \mathcal{G}(\eta)^{-1} \mathcal{L}(v) \dot v)  \,\mathrm{d}x_1. \\
\end{split}
\end{align}
Using the fact that $\mathcal{G}(\eta)$ and $\mathcal{G}(\eta)^{-1}$ are self-adjoint operators, the integral is equal to
\begin{align*}
	&\int_{\mathbb{R}} \Big[ -\mathcal{L}(v) \dot v (\dot\varphi - \mathcal{G}(\eta)^{-1} \mathcal{L}(v) \dot v) + (\dot\varphi - \mathcal{G}(\eta)^{-1} \mathcal{L}(v) \dot v) \mathcal{G}(\eta) \dot\varphi \Big] \,\mathrm{d}x_1 \\
	& \qquad = \int_{\mathbb{R}} \mathcal{G}(\eta)(\dot\varphi - \mathcal{G}(\eta)^{-1} \mathcal{L}(v) \dot v) (\dot\varphi - \mathcal{G}(\eta)^{-1} \mathcal{L}(v) \dot v) \,\mathrm{d}x_1.
\end{align*}
Substituting this into the equation \eqref{D2 E_c Lemma Extension} yields our desired result.  
\end{proof}

We finish this subsection by showing that Assumption~\ref{Assumption spectrum} is satisfied.

\begin{thm}[Spectrum] \label{Theorem spectrum I inverse Hc}
Fix any choice of $0 < \rho_0 < a_0$ subject to the compatibility condition \eqref{equation eta phi x_bar y_bar}, and consider the family of traveling wave solutions $\{U_c\}_{c \in \mathcal{I}}$ as in \eqref{Uc:with_I}.  Then for all $c \in \mathcal{I}$, $I^{-1}H_c$ has one negative eigenvalue, $0$ is in the spectrum, and the rest of the spectrum $\Sigma_c \subset (0,\infty)$ is bounded away from $0$.
\end{thm}

\begin{proof}
From the asymptotic information furnished by the existence theorem \eqref{asymptotic forms traveling wave soln Uc}, we infer that
\[
	\mathfrak{a}_1=O(\epsilon^3), \quad
	\mathfrak{a}_2 = O(\epsilon^3), \quad
	\mathfrak{b}_1 = \mathfrak{a}_1 - c + o(\epsilon^2) = O(\epsilon), \quad
	\mathfrak{b}_2 = \mathfrak{a}_2 + o(\epsilon^2) = O(\epsilon).
\]
Then from Lemmas~\ref{Lemma D^2 V_c^aug quadratic form} and \ref{Lemma extension D^2 Ec}, we can write 
\[
	H_c =
	\begin{pmatrix}
	g - b\partial_{x_1}^2  &  0  &  0  \\
	0  &  |\partial_{x_1}|  &  0  \\
	0  &  0  &  \mathcal{A}
	\end{pmatrix}
	+ O(\epsilon^3) \in \mathrm{Lin}(\mathbb{X}, \mathbb{X}^*),
\]
where
\[
	\mathcal{A} :=
	\begin{pmatrix} 
		\mathcal{A}_{33}  &  \mathcal{A}_{34}  \\  
		\mathcal{A}_{43}  &  \mathcal{A}_{44} \end{pmatrix}
	= \frac{\epsilon^2}{4\pi}
	\begin{pmatrix}
	-\alpha & 0                 & \alpha  & 0 \\
	0       & \delta_1 + \alpha & 0       & -\beta \\
	\alpha  & 0                 & -\alpha & 0 \\
	0       & -\beta            & 0       & \delta_2 + \alpha
	\end{pmatrix},
\]
and
\[
	\alpha := \frac{\gamma_1 \gamma_2}{2} \left(\frac{1}{\rho^2} - \frac{1}{a^2} \right), \quad
	\beta := \frac{\gamma_1 \gamma_2}{2} \left(\frac{1}{\rho^2} + \frac{1}{a^2} \right), \quad
	\delta_1 := \frac{\gamma_1^2}{(a-\rho)^2}, \quad
	\delta_2 := \frac{\gamma_2^2}{(a+\rho)^2}.
\]
Setting $\epsilon = 0$, it follows that $H_c$ has a zero eigenvalue of multiplicity $4$, and the remainder of the spectrum is strictly positive.  Thus, when $0 < |\epsilon| \ll 1$, $I^{-1}H_c$ will have positive spectrum $\Sigma_c \subset (0,\infty)$ along with four eigenvalues bifurcating from $0$.   

To determine these, we look more closely at the matrix $\mathcal{A}$.  In particular, direct computation confirms that it has the eigenvalues:
\[
	0, \quad
	-2\alpha, \quad
	\frac{2\alpha + \delta_1 + \delta_2 + \sqrt{(\delta_1-\delta_2)^2+4\beta^2}}{2}, \quad
	\frac{2\alpha + \delta_1 + \delta_2 - \sqrt{(\delta_1-\delta_2)^2+4\beta^2}}{2}.
\]
We know that $0$ is in the spectrum of $I^{-1} H_c$ due to translation invariance.  Clearly, $-2\alpha < 0$, and the third eigenvalue above is positive.  We claim that the last eigenvalue is also positive.  Indeed,
\[
	2\alpha + \delta_1 + \delta_2 - \sqrt{(\delta_1-\delta_2)^2+4\beta^2} > 0
\]
is equivalent to
\[
	\alpha^2 + \alpha\delta_1 + \alpha\delta_2 + \delta_1\delta_2 - \beta^2 > 0.
\]
Using the compatibility condition \eqref{eqn:compatibility condition}, we compute
\begin{align*}
	\alpha^2 + \alpha\delta_1 + \alpha\delta_2 + \delta_1\delta_2 - \beta^2 &=
	\frac{\gamma_1 \gamma_2}{2} \left( \frac{1}{\rho^2} - \frac{1}{a^2} \right) \left( \frac{\gamma_1^2}{(a-\rho)^2} + \frac{\gamma_2^2}{(a+\rho)^2} \right) - \frac{\gamma_1^2 \gamma_2^2}{a^2 \rho^2} \\
	&= \frac{2(a+\rho)}{(a-\rho)(a^2+a\rho+\rho^2)^2} \gamma_1^4 > 0.
\end{align*}
  We then conclude that, for $|\epsilon| > 0$ sufficiently small, the spectrum of $I^{-1}H_c$ consists of precisely one negative eigenvalue, one zero eigenvalue, and the rest is positive.  
\end{proof}

We have verified all of the assumptions \ref{Assumption space}--\ref{Assumption spectrum} of Varholm--Wahl\'en--Walsh instability theory.  The next subsection shows that $d''(c) < 0$, which implies orbital instability.

\subsection{Proof of Theorem~\ref{Theorem Stability fixed dipole}}
\label{subsection proof of stability (main) theorem}

Using the expressions for the momentum $P$, $\bar{x}_2$, and $\bar{y}_2$, we can compute:
\begin{align*}
	d'(c) &= \epsilon \gamma_1 (-a+\rho) - \epsilon \gamma_2 (-a-\rho) - \int_{\mathbb{R}} \eta(\varphi' + \epsilon\nabla_\top\Theta) \,\mathrm{d}x_1.
\end{align*}
Differentiating once more yields
\begin{align*}
	d''(c) &=  \epsilon \gamma_1 \partial_c  (-a  + \rho) + \epsilon \gamma_2 \partial_c(a+\rho) - \int_{\mathbb{R}} \Big( (\partial_c\eta) (\varphi' + \epsilon\nabla_\top\Theta) + \eta \; \partial_c (\varphi' + \epsilon \nabla_\top\Theta) \Big) \,\mathrm{d}x_1.
\end{align*}
Recalling the definition of $\mathscr{T}$ in \eqref{matrix T}, using the compatibility \eqref{eqn:compatibility condition} and variations for $a$ and $\rho$ in \eqref{a_tilde c rho_tilde c 2 point vortices}, 
we obtain
\begin{align*}
	d''(c) &= -\gamma_1 (a_{\tilde{c}} - \rho_{\tilde{c}}) + \gamma_2 (a_{\tilde{c}} + \rho_{\tilde{c}}) + O(\epsilon^3) \\
	&= -\frac{\gamma_1}{\det\mathscr{T}} \left( -\frac{\gamma_2^0}{2\pi(a_0+\rho_0)^2} + \frac{-\gamma_1^0+\gamma_2^0}{4\pi\rho_0^2} + \frac{\gamma_1^0+\gamma_2^0}{4\pi a_0^2} \right) \\
	&\qquad + \frac{\gamma_2}{\det\mathscr{T}} \left( \frac{\gamma_1^0}{2\pi(a_0-\rho_0)^2} + \frac{-\gamma_1^0+\gamma_2^0}{4\pi\rho_0^2} - \frac{\gamma_1^0+\gamma_2^0}{4\pi a_0^2} \right) + O(\epsilon^3) \\
	&= \frac{\gamma_1^2}{2\pi \det\mathscr{T}} \frac{6 a_0 \rho_0^2}{(a_0+\rho_0)(a_0-\rho_0)^2(a_0^2+a_0\rho_0+\rho_0^2)} + O(\epsilon^3).
\end{align*}
Thus, since $\det\mathscr{T} < 0$, we conclude that $d''(c) < 0$ for $|\epsilon| \ll 1$ and $c=O(\epsilon)$. Hence, Theorem~\ref{Therem VWW instability} tells us that the corresponding water waves $\{U_c\}$ constructed in Theorem~\ref{Theorem existence waves 2 point votices phantom diff vor str} are orbitally unstable.

\section*{Acknowledgements}
This work was supported in part by the National Science Foundation through DMS-1549934 and DMS-1710989.

The author would like to express his sincere gratitude to \thanks{Samuel Walsh} for his continuous support and insightful comments.  The author is also grateful to Roberto Camassa who suggested this problem.


\appendix

\section{Abstract instability theory}
\label{Appendix Varholm--Wahlen--Walsh instability theory}

This section summarizes the instability theory developed by Varholm, Wahl\'en, and Walsh in \cite[Sections 2, 4]{Varholm_Wahlen_Walsh2018}.  We are considering the stability property of an abstract Hamiltonian
\begin{equation} \label{eqn Appendix abstract Hamiltonian}
	\frac{\mathrm{d}u}{\mathrm{d}t} = J(u) DE(u), \quad u|_{t=0} = u_0,
\end{equation}
where $J$ is the Poisson map and $E$ is the energy.  Let $\mathbb{X}$ be a Hilbert space, and $\mathbb{V}$ and $\mathbb{W}$ be reflexive Banach spaces such that
\[
	\mathbb{W} \hookrightarrow \mathbb{V} \hookrightarrow \mathbb{X}.
\]
Let $\mathbb{X}^*$ be the (continuous) dual space of $\mathbb{X}$, which is naturally isomorphic to $\mathbb{X}$ via the mapping $I:\mathbb{X} \to \mathbb{X}^*$. 

\begin{assumption}[Spaces] \label{Assumption space}
There exist constants $\theta \in (0,1]$ and $C>0$ such that
\begin{equation} \label{Inequality interpolation X V W}
	\|u\|_{\mathbb{V}}^3 \le C \|u\|_{\mathbb{X}}^{2+\theta} \|u\|_{\mathbb{W}}^{1-\theta}
\end{equation}
for all $u \in \mathbb{W}$.
\end{assumption}

Let $\mathcal{O} \subset \mathbb{X}$ be an open set.  Suppose that
\[
	J(u) := B(u) \widehat{J},
\]
where for each $u \in \mathcal{O} \cap \mathbb{V}$, $B(u)$ is a bounded linear operator in $\mathbb{X}$, that is, $B(u) \in \mathrm{Lin}(\mathbb{X})$, and $\widehat{J}:\mathcal{D}(J) \subset \mathbb{X}^* \to \mathbb{X}$ is a closed linear operator.

\begin{assumption}[Poisson map] \label{Assumption Poisson map}
\hfill 
\begin{enumerate} [label=(\roman*)]
	\item The domain $\mathcal{D}(\widehat{J})$ is dense in $\mathbb{X}^*$.
	
	\item $\widehat{J}$ is injective.
	
	\item For each $u \in \mathcal{O} \cap \mathbb{V}$, the operator $B(u)$ is bijective.
	
	\item The map $u \mapsto B(u)$ is of class $C^1(\mathcal{O} \cap \mathbb{V}; \mathrm{Lin}(\mathbb{X})) \cap C^1(\mathcal{O} \cap \mathbb{W}; \mathrm{Lin}(\mathbb{W}))$.
	
	\item For each $u \in \mathcal{O} \cap \mathbb{V}$, $J(u)$ is skew-adjoint in the sense that
	\[
		\langle J(u)v, w \rangle = -\langle v, J(u)w \rangle
	\]
	for all $v, w \in \mathcal{D}(\widehat{J})$.
\end{enumerate}
\end{assumption}

We suppose that $\mathbb{V}$ is chosen so that $E \in C^3(\mathcal{O}\cap \mathbb{V}; \mathbb{R})$.  In addition, assume that there exists a momentum functional $P \in C^3(\mathcal{O} \cap \mathbb{V}; \mathbb{X})$, and that both it and the energy are conserved by solutions of \eqref{eqn Appendix abstract Hamiltonian}.

\begin{assumption}[Derivative extension] \label{Assumption derivative extension}
There exist mappings $\nabla E$, $\nabla P \in C^0(\mathcal{O} \cap \mathbb{V}; \mathbb{X}^*)$ such that $\nabla E(u)$ and $\nabla P(u)$ are extensions of $DE(u)$ and $DP(u)$, respectively, for every $u \in \mathcal{O} \cap \mathbb{V}$.
\end{assumption}

Suppose that there is a one-parameter family of affine maps $T(s): \mathbb{X} \to \mathbb{X}$, with the linear part $dT(s) := T(s)u - T(s)0$ having the properties:

\begin{assumption}[Symmetry group] \label{Assumption symmetry group}
The symmetry group $T(\cdot)$ satisfies the following.
\begin{enumerate} [label=(\roman*)]
	\item \emph{(Invariance)} The neighborhood $\mathcal{O}$, and the subspaces $\mathbb{V}$ and $\mathbb{W}$, are all invariant under the symmetry group.  Moreover, $I^{-1}\mathcal{D}(\widehat{J})$ is invariant under the \emph{linear} symmetry group, or equivalently, $\mathcal{D}(\widehat{J})$ is invariant under the adjoint $dT^*(s): \mathbb{X}^* \to \mathbb{X}^*$.
	
	\item \emph{(Flow property)} We have $T(0)=dT(0)=\mathrm{Id}_{\mathbb{X}}$, and for all $s, r \in \mathbb{R}$,
	\[
		T(s+r) = T(s)T(r), \quad \text{and hence} \quad dT(s+r)=dT(s)dT(r).
	\]
	
	\item \emph{(Unitary)} The linear part $dT(s)$ is a unitary operator on $\mathbb{X}$ for each $s \in \mathbb{R}$, or equivalently,
	\begin{equation} \label{eqn assumption Symmetry-unitary}
		dT^*(s)I = I \, dT(-s), \quad \text{for all } s \in \mathbb{R}.
	\end{equation}
	Moreover, the linear part is an isometry on the spaces $\mathbb{V}$ and $\mathbb{W}$.
	
	\item \emph{(Strong continuity)} The symmetry group is strongly continuous on both $\mathbb{X}$, $\mathbb{V}$, and $\mathbb{W}$.
	
	\item \emph{(Affine part)} The function $T(\cdot)0$ belongs to $C^3(\mathbb{R};\mathbb{W})$ and there exists an increasing function $w:[0,\infty) \to [0,\infty)$ such that
	\[
		\|T(s)0\|_{\mathbb{W}} \ne w(\|T(s)0\|_{\mathbb{X}}), \quad \text{for all } s \in \mathbb{R}.
	\]
	
	\item \emph{(Commutativity with $J$)} For all $s \in \mathbb{R}$,
	\begin{equation} \begin{gathered}
	\label{eqn assumption Symmetry-commutativity}
		\widehat{J}I \, dT(s) = dT(s) \widehat{J}I, \\
		dT(s) B(u) = B(T(s)u) dT(s), \quad \text{for all } u \in \mathcal{O} \cap \mathbb{V}.
	\end{gathered} \end{equation}
	
	\item \emph{(Infinitesimal generator)} The infinitesimal generator of $T$ is the affine mapping
	\[
		T'(0)u = \lim_{s \to 0} \left( s^{-1}(T(s)u-u) \right) = dT'(0) + T'(0)0,
	\]
	with dense domain $\mathcal{D}(T'(0)) \subset \mathbb{X}$ consisting of all $u \in \mathbb{X}$ such that the limit exists in $\mathbb{X}$.  Similarly, we may speak of the dense subspaces $\mathcal{D}(T'(0)|_{\mathbb{V}}) \subset \mathbb{V}$ and $\mathcal{D}(T'(0)|_{\mathbb{W}}) \subset \mathbb{W}$ on which the limit exists in $\mathbb{V}$ and $\mathbb{W}$, respectively.  We assume that $\nabla P(u) \in \mathcal{D}(\widehat{J})$ for every $u \in \mathcal{D}(T'(0)|_{\mathbb{V}}) \cap \mathcal{O}$, and that
	\begin{equation} \label{eqn assumtion Symmetry-infinitesimal generator 1}
		T'(0)u = J(u) \nabla P(u)
	\end{equation}
	for all such $u$.  Moreover, we assume that
	\begin{equation} \label{eqn assumtion Symmetry-infinitesimal generator 2}
		\widehat{J}I \, dT'(0) = dT'(0) \widehat{J}I.
	\end{equation}
	
	\item \emph{(Density)} The subspace
	\[
		\mathcal{D}(T'(0)|_{\mathbb{W}}) \cap \mathrm{Rng} \, \widehat{J}
	\]
	is dense in $\mathbb{X}$.
	
	\item \emph{(Conservation)} For all $u \in \mathcal{O} \cap \mathbb{V}$, the energy is conserved by flow of the symmetry group:
	\begin{equation} \label{eqn assumption Symmetry-conservation}
		E(u) = E(T(s)u), \quad \text{for all } s \in \mathbb{R}.
	\end{equation}
\end{enumerate}
\end{assumption}

We say $u \in C^1(\mathbb{R}; \mathcal{O} \cap \mathbb{W})$ is a \emph{bound state} of the Hamiltonian system \eqref{eqn Appendix abstract Hamiltonian} if $u$ is a solution of the form
\[
	u(t) = T(ct)U,
\]
for some $c \in \mathbb{R}$ and $U \in \mathcal{O} \cap \mathbb{W}$.

\begin{assumption}[Bound states] \label{Assumption bound states}
There exists a one-parameter family of bound state solutions $\{ U_c: c \in \mathcal{I} \}$ to the Hamiltonian system \eqref{eqn Appendix abstract Hamiltonian}.
\begin{enumerate} [label=(\roman*)]
	\item The mapping $c \in \mathcal{I} \mapsto U_c \in \mathcal{O} \cap \mathbb{W}$ is $C^1$.
	
	\item The non-degeneracy condition $T'(0)U_c \ne 0$ holds for every $c \in \mathcal{I}$.  Equivalently, $U_c$ is never a critical point of the momentum.
	
	\item For all $c \in \mathcal{I}$,
	\begin{equation} \label{eqn assumption bound states iii 1}
		U_c \in \mathcal{D}(T'''(0)) \cap \mathcal{D}(\widehat{J}IT'(0)),
	\end{equation}
	and
	\begin{equation} \label{eqn assumption bound states iii 2}
		\widehat{J}IT'(0)U_c \in \mathcal{D}(T'(0)|_{\mathbb{W}}).
	\end{equation}
	
	\item Either $s \mapsto T(s)U_c$ is periodic, or $\liminf_{|s| \to \infty} \|T(s)U_c - U_c\|_{\mathbb{X}} > 0$.
\end{enumerate}
\end{assumption}

Define $E_c(u) := E(u) - cP(u)$ to be the \emph{augmented Hamiltonian}.  Then we have the following assumption:

\begin{assumption}[Spectrum] \label{Assumption spectrum}
The mapping
\[
	u \in \mathbb{V} \mapsto \left\langle D^2 E_c(U_c)u, \cdot \right\rangle_{\mathbb{V}^* \times \mathbb{V}} \in \mathbb{V}^*
\]
extends uniquely to a bounded linear operator $H_c: \mathbb{X} \to \mathbb{X}^*$ with the following properties
\begin{enumerate}[label=(\roman*)]
	\item $I^{-1}H_c$ is self-adjoint on $\mathbb{X}$.
	
	\item The eigenvalues of $I^{-1}H_c$ satisfy
	\[
		\mathrm{spec}(I^{-1}H_c) = \{-\mu_c^2\} \cup \{0\} \cup \Sigma,
	\]
	where $-\mu_c^2 < 0$ is a simple eigenvalue corresponding to a unit eigenvector $\chi_c$, 0 is a simple eigenvalue generated by $T$, and $\Sigma$ is a subset of the positive real axis bounded away from 0.
\end{enumerate}
\end{assumption}

\begin{assumption}[Local existence] \label{Assumption local existence}
There exists $\nu_0 > 0$ and $t_0 > 0$ such that for all initial data
$u_0 \in U_{\nu_0}$, there exists a unique solution to the ODE \eqref{eqn Appendix abstract Hamiltonian} on the interval $[0, t_0)$.
\end{assumption}

Let $d(c) := E_c(U_c) = E(U_c) - cP(U_c)$ be the \emph{moment of instability}, where $U_c$ is a traveling wave.  Moreover, for each $\rho>0$, the tubular neighborhood of radius $\rho$ in $\mathbb{X}$ for the $U_c$-orbit generated by $T$ is
\begin{equation} \label{defn Appendix tubular neighborhood}
	\mathcal{U}_\rho := \left\{ u \in \mathcal{O}: \inf_{s\in\mathbb{R}} \|T(s)U_c-u\|_{\mathbb{W}} < \rho \right\},
\end{equation}
We have the following instability theorem; see \cite[Theorem 2.6]{Varholm_Wahlen_Walsh2018}.

\begin{thm} [Instability] \label{Therem VWW instability}
Suppose that all assumptions \ref{Assumption space}--\ref{Assumption local existence} are satisfied, and that there exists a family of traveling water waves $U_c$.  Then if $d''(c) < 0$, the traveling wave $U_c$ is orbitally unstable.  That is, there exists a $\nu_0 > 0$ such that for every $0 < \nu < \nu_0$ there exists initial data in $\mathcal{U}_{\nu}$ whose corresponding solution exits $\mathcal{U}_{\nu_0}$ in finite time.
\end{thm}

\section{Steady and unsteady equations}
\label{Appendix steady and unsteady equations}

For the convenience of the reader, in this appendix we derive the nonlocal formulations for the water wave with a finite dipole problem \eqref{equation eta phi x_bar y_bar}.

Using the definitions of $\varphi$ in \eqref{defn varphi} and $\mathcal{G}(\eta)$ in \eqref{Dirichlet Neumann operator}, we obtain
\begin{equation} \label{relation Phi a1 and a2}
	\nabla \Phi = \frac{1}{\langle \eta' \rangle^2}
	\begin{pmatrix} 1 & -\eta' \\ \eta' & 1 \end{pmatrix}
	\begin{pmatrix}	\varphi' \\ \mathcal{G}(\eta) \varphi \end{pmatrix} =
	\frac{1}{\langle \eta' \rangle^2}
	\begin{pmatrix}	\varphi' - \eta' \mathcal{G}(\eta) \varphi \\ \eta' \varphi' + \mathcal{G}(\eta) \varphi \end{pmatrix}.
\end{equation}
Combining with the definitions of $\psi$ in \eqref{defn psi} gives
\begin{align} \label{relationship between varphi and psi}
	\begin{pmatrix} \mathcal{G}(\eta)\varphi \\ \varphi' \end{pmatrix} = \begin{pmatrix} \psi' \\ - \mathcal{G}(\eta)\psi \end{pmatrix}.
\end{align}
Then from the incompressible Euler equation \eqref{incompressible Euler eqn}, we can derive the \emph{unsteady equation for velocity potential} on $S$
\begin{multline} \label{unsteady velocity potential equation}
	\partial_t \varphi = -\frac{1}{2\langle\eta'\rangle^2} \left( (\varphi')^2 - 2\eta'\varphi'\mathcal{G}(\eta)\varphi - (\mathcal{G}(\eta)\varphi)^2 \right) - \epsilon \partial_t \Theta + \epsilon \varphi' \partial_{x_2}\Gamma - \frac{\epsilon^2}{2} |\nabla\Gamma|^2 \\
	- \eta + b \frac{\eta''}{\langle \eta' \rangle^3}.
\end{multline}
Using the relation \eqref{relationship between varphi and psi}, we also have the \emph{unsteady equation for stream function} on $S$:
\begin{multline} \label{unsteady stream function equation}
	\partial_t \varphi = -\frac{1}{2\langle\eta'\rangle^2} \left( (\mathcal{G}(\eta)\psi)^2 + 2\eta'\psi'\mathcal{G}(\eta)\psi - (\psi')^2 \right) - \epsilon \partial_t \Theta - \epsilon \mathcal{G}(\eta)\psi \, \partial_{x_2}\Gamma - \frac{\epsilon^2}{2} |\nabla\Gamma|^2 \\
	- \eta + b \frac{\eta''}{\langle \eta' \rangle^3}.
\end{multline}
For the traveling water waves, the \emph{steady equation for velocity potential} on $S$ is
\begin{multline} \label{steady equation varphi veloctiy potential}
	-\frac{c}{\langle\eta'\rangle^2} (\varphi' - \eta' \mathcal{G}(\eta)\varphi) + c \epsilon \partial_{x_2}\Gamma + \frac{1}{2\langle\eta'\rangle^2} \left[ (\varphi')^2 + (\mathcal{G}(\eta)\varphi)^2 \right] \\
	+ \frac{\epsilon}{\langle\eta'\rangle^2} \left[ -\varphi'\nabla_\perp\Gamma + \mathcal{G}(\eta)\varphi \; \nabla_\top\Gamma \right] + \frac{\epsilon^2}{2} |\nabla\Gamma|^2 + \eta - b \frac{\eta''}{\langle \eta' \rangle^3} = 0,
\end{multline}
and the \emph{steady equation for stream function} on $S$ is:
\begin{multline} \label{steady equation psi stream function}
	\frac{c}{\langle\eta'\rangle^2} (\psi' + \eta' \mathcal{G}(\eta)\psi) + c \epsilon \partial_{x_2}\Gamma + \frac{1}{2\langle\eta'\rangle^2} \left[ (\psi')^2 + (\mathcal{G}(\eta)\psi)^2 \right] \\
	+ \frac{\epsilon}{\langle\eta'\rangle^2} \left[ \mathcal{G}(\eta)\psi \nabla_\perp\Gamma + \psi' \; \nabla_\top\Gamma \right]	+ \frac{\epsilon^2}{2} |\nabla\Gamma|^2 + \eta - b \frac{\eta''}{\langle \eta' \rangle^3} = 0.
\end{multline}

\section{Variations of the energy and momentum}
\label{Appendix variations of E and P}

Finally, in this appendix we record the first and second Fr\'echet derivatives of the energy and momentum.  

Recall that 
\[
	\mathfrak{a}=(\nabla(\mathcal{H}\varphi))|_{S_t}, \quad  \xi=(\Upsilon_{1_{x_1}},\Xi_{1_{x_2}})^T, \quad \textrm{and} \quad  \zeta=(\Upsilon_{2_{x_1}},\Xi_{2_{x_2}})^T.
\]
Let $\nabla\xi:=(\Upsilon_{1_{x_1 x_1}},\Xi_{1_{x_2 x_2}})^T$, $\nabla\zeta:=(\Upsilon_{2_{x_1 x_1}},\Xi_{2_{x_2 x_2}})^T$, and
\[
	D^2_{\bar x} \Theta := 
	\begin{pmatrix} 
		\Upsilon_{1_{x_1 x_1}} & \Xi_{1_{x_1 x_2}} \\
		\Xi_{1_{x_1 x_2}}      & \Upsilon_{1_{x_2 x_2}}
	\end{pmatrix},
	\quad \mathrm{and} \quad
	D^2_{\bar y} \Theta := 
	\begin{pmatrix} 
	\Upsilon_{2_{x_1 x_1}} & \Xi_{2_{x_1 x_2}} \\
	\Xi_{2_{x_1 x_2}}      & \Upsilon_{2_{x_2 x_2}}
	\end{pmatrix}.
\]

\subsection*{Variations of $K_0(u)$}

We compute that
\[
	D_\varphi K_0(u) \dot{\varphi} = \int_\mathbb{R} \dot{\varphi} \mathcal{G}(\eta)\varphi \,\mathrm{d}x_1, \quad
	D_\eta K_0(u) \dot{\eta} = \frac{1}{2} \int_\mathbb{R} \varphi \langle D_\eta\mathcal{G}(\eta)\dot{\eta}, \varphi \rangle \,\mathrm{d}x_1,
\]
and 
\[
	\langle D_\varphi^2 K_0(u)\dot{\varphi}, \dot{\varphi} \rangle = \int_\mathbb{R} \dot{\varphi}\mathcal{G}(\eta)\dot{\varphi} \,\mathrm{d}x_1,
\]
\[
	\langle D_\varphi D_\eta K_0(u)\dot{\varphi}, \dot{\eta} \rangle = \int_\mathbb{R} \dot{\varphi} \langle D_\eta \mathcal{G}(\eta)\dot{\eta},\varphi \rangle \,\mathrm{d}x_1 = \int_{\mathbb{R}} \dot{\eta} (\mathfrak{a}_1 \dot{\varphi}' - \mathfrak{a}_2 \mathcal{G}(\eta)\dot{\varphi}) \,\mathrm{d}x_1
\]
\[
	\langle D_\eta^2 K_0(u)\dot{\eta}, \dot{\eta} \rangle = \frac{1}{2} \int_\mathbb{R} \varphi \langle \langle D_\eta^2 \mathcal{G}(\eta)\dot{\eta}, \dot{\eta} \rangle, \varphi \rangle \,\mathrm{d}x_1 = \int_{\mathbb{R}} (\mathfrak{a}_1'\mathfrak{a}_2\dot{\eta}^2 + \mathfrak{a}_2\dot{\eta}\mathcal{G}(\eta)(\mathfrak{a}_2\dot{\eta})) \,\mathrm{d}x_1.
\]

\subsection*{Variations of $K_1(u)$}
Likewise, the first variations of $K_1$ are
\begin{align*}
	D_\varphi K_1(u) \dot{\varphi} = \int_\mathbb{R} \dot{\varphi} \nabla_\perp \Theta \,\mathrm{d}x_1, \qquad
	D_\eta K_1(u) \dot{\eta} = \int_\mathbb{R} \dot\eta \varphi' \Theta_{x_1}|_S \,\mathrm{d}x_1,
\end{align*}
\begin{align*}
	\nabla_{\bar x} K_1(u) = -\int_\mathbb{R} \varphi \nabla_\perp \xi \,\mathrm{d}x_1,
	\qquad
	\nabla_{\bar y} K_1(u) = -\int_\mathbb{R} \varphi \nabla_\perp \zeta \,\mathrm{d}x_1,
\end{align*}
and the second are given by
\[
	\langle D_\varphi D_\eta K_1(u)\dot{\eta}, \dot{\varphi} \rangle = \int_\mathbb{R} \dot{\eta}\dot{\varphi}' \Theta_{x_1}|_S \,\mathrm{d}x_1, \qquad
	\langle D^2_\eta K_1(u)\dot{\eta}, \dot{\eta} \rangle = \int_\mathbb{R} \dot{\eta}^2 \varphi' \Theta_{x_1 x_2}|_S \,\mathrm{d}x_1,
\]
\[
	D^2_{\bar x} K_1(u) = \int_{\mathbb{R}} \varphi \nabla_\perp D^2_{\bar x}\Theta \,\mathrm{d}x_1, \qquad
	D^2_{\bar y} K_1(u) = \int_{\mathbb{R}} \varphi \nabla_\perp D^2_{\bar y}\Theta \,\mathrm{d}x_1,
\]
\[
	\nabla_{\bar x} D_{\eta} K_1(u) \dot{\eta} = -\int_{\mathbb{R}} \dot{\eta}\varphi' (\nabla\xi)|_S \,\mathrm{d}x_1,
	\qquad
	\nabla_{\bar y} D_{\eta} K_1(u) \dot{\eta} = -\int_{\mathbb{R}} \dot{\eta}\varphi' (\nabla\zeta)|_S \,\mathrm{d}x_1,
\]
\[
	\nabla_{\bar x} D_\varphi K_1(u)\dot{\varphi} = -\int_\mathbb{R} \dot{\varphi} \nabla_\perp \xi \,\mathrm{d}x_1, 
	\qquad
	\nabla_{\bar y} D_\varphi K_1(u)\dot{\varphi} = -\int_\mathbb{R} \dot{\varphi} \nabla_\perp \zeta \,\mathrm{d}x_1.
\]

\subsection*{Variations of $K_2(u)$}
It is straightforward to compute that 
\begin{align*}
	D_\eta K_2(u)\dot{\eta} &= \frac{1}{2} \int_\mathbb{R} \dot{\eta} |(\nabla\Theta)|_S|^2 \,\mathrm{d}x_1,
\end{align*}
\[
	\nabla_{\bar x} K_2(u) = \nabla_{\bar x}\Gamma^* - \frac{1}{2} \int_{\mathbb{R}} \nabla_\perp(\xi\Theta) \,\mathrm{d}x_1, \quad
	\nabla_{\bar y} K_2(u) = \nabla_{\bar y}\Gamma^* - \frac{1}{2} \int_{\mathbb{R}} \nabla_\perp(\zeta\Theta) \,\mathrm{d}x_1, 
\]
and
\[
	\langle D^2_\eta K_2(u)\dot{\eta},\dot{\eta} \rangle = \int_\mathbb{R} \dot{\eta}^2 \Big( \Theta_{x_1}\Theta_{x_1 x_2} + \Theta_{x_2}\Theta_{x_2 x_2} \Big)\Big|_S \,\mathrm{d}x_1,
\]
\[
	D^2_{\bar x} K_2(u) = 2D^2_{\bar x} \Gamma^* + \frac{1}{2} \int_{\mathbb{R}} \nabla_\perp (\Theta D^2_{\bar x}\Theta + \xi\xi^T) \,\mathrm{d}x_1,
\]
\[
	D^2_{\bar y} K_2(u) = 2D^2_{\bar y} \Gamma^* + \frac{1}{2} \int_{\mathbb{R}} \nabla_\perp (\Theta D^2_{\bar y}\Theta + \zeta\zeta^T) \,\mathrm{d}x_1,
\]
\[
	\nabla_{\bar x}\nabla_{\bar y} K_2(u) = \nabla_{\bar x}\nabla_{\bar y} \Gamma^* + \frac{1}{2} \int_{\mathbb{R}} \nabla_\perp(\xi \odot \zeta) \,\mathrm{d}x_1,
\]
\[
	\nabla_{\bar x} D_\eta K_2(u)\dot\eta = -\int_{\mathbb{R}} \dot\eta ((D_x \xi)\nabla\Theta)|_S \,\mathrm{d}x_1, \quad
	\nabla_{\bar y} D_\eta K_2(u)\dot\eta = -\int_{\mathbb{R}} \dot\eta ((D_x \zeta)\nabla\Theta)|_S \,\mathrm{d}x_1.
\]

\subsection*{Variations of $V(u)$}
Similarly, we find that 
\begin{align*}
	D_\eta V(u)\dot{\eta} &= \int_\mathbb{R} \dot{\eta} \left( g\eta - b \frac{\eta''}{\langle\eta'\rangle^3} \right) \,\mathrm{d}x_1, \\
	\langle D_{\eta}^2 V(u)\dot{\eta}, \hat{\eta} \rangle &= \int_{\mathbb{R}} \left( g\hat{\eta}\dot{\eta} + \frac{b}{\langle\eta'\rangle^3} \hat{\eta}'\dot{\eta}' \right) \,\mathrm{d}x_1.
\end{align*}

\subsection*{Variations of $P(u)$}

Finally, the first variations of momentum $P(u)$ are given in Section~\ref{subsection Hamiltonian formulation}.  The second derivatives are as follows:
\[
	\langle D_\eta D_\varphi P(u)\dot{\varphi},\dot{\eta} \rangle = -\int_\mathbb{R} \dot{\eta}'\dot{\varphi} \,\mathrm{d}x_1, \qquad
	\langle D^2_\eta P(u)\dot{\eta},\dot{\eta} \rangle = \epsilon \int_\mathbb{R} \dot{\eta}^2 \Theta_{x_1 x_2}|_S \,\mathrm{d}x_1,
\]
\[
	D^2_{\bar x} P(u) = -\epsilon\int_{\mathbb{R}} \eta' (D^2_{\bar x} \Theta)|_S \,\mathrm{d}x_1, \qquad
	D^2_{\bar y} P(u) = -\epsilon\int_{\mathbb{R}} \eta' (D^2_{\bar y} \Theta)|_S \,\mathrm{d}x_1,
\]
\[
	\nabla_{\bar x} D_\eta P(u)\dot{\eta} = -\epsilon\int_\mathbb{R} \dot{\eta} (\nabla\xi)|_S \,\mathrm{d}x_1, \qquad
	\nabla_{\bar y} D_\eta P(u)\dot{\eta} = -\epsilon\int_\mathbb{R} \dot{\eta} (\nabla\zeta)|_S \,\mathrm{d}x_1.
\]

\bibliography{Le_references}
\bibliographystyle{plain}

\end{document}